\documentclass[a4paper,10pt]{amsart}
\usepackage{graphicx} 
\usepackage{eso-pic}
\usepackage{amsmath, amsthm, amsfonts, amssymb, amscd}
\usepackage[T1]{fontenc}
\usepackage[utf8x]{inputenc}
\usepackage{wrapfig}

\usepackage{hyperref}

\newtheorem{Teo}{Theorem}[section]
\newtheorem{Def}{Definition}[section]

\newtheorem{Lema}{Lemma}[section]
\newtheorem{Prop}{Proposition}[section]
\newtheorem{Cor}{Corollary}[section]
\newtheorem{Remark}{Remark}[section]
\newtheorem{maintheorem}{Theorem}

\author{I. Rios }
\address{Instituto de Matem\'atica e Estat\'\i stica, Universidade Federal Fluminense}
\email{rios@mat.uff.br}

\author{J. Siqueira}
\address{Centro de Matem\'atica da Universidade do Porto}
\email{jaqueline.rocha@fc.up.pt}

\title{On equilibrium states for partially hyperbolic horsehoes}

\begin{document}

\begin{abstract}

We prove existence and uniqueness of equilibrium states for a family of partially hyperbolic 
systems, with respect to  H\"older continuous potentials with small variation. The family 
comes from the projection,  on the center-unstable direction, of a family of partially hyperbolic horseshoes   
introduced in [Díaz,L., Horita,V., Rios, I., Sambarino, M., {\em Destroying horseshoes via heterodimensional cycles: generating bifurcations inside homoclinic classes}, Ergodic Theory and Dynamical Systems, 2009]. For the original three dimensional system, we consider potentials with small variation, constant on local stable manifolds, obtaining existence and uniqueness of equilibrium states.

\end{abstract}

\maketitle

\pagenumbering{arabic}

%%%%%%%%%%%%%%%%%%%%%%%%%%%%%%%%%%%%%%%%%%%%%%%%%%%%%%

%%%%%%%%%%%%%%%%%%%%%%%%%

%   Introduction                                                            %

%%%%%%%%%%%%%%%%%%%%%%%%%%%%%%%%%%%%%%%%%%%%%%%%%%%%%%

%%%%%%%%%%%%%%%%%%%%%%%%%

%\bibliographystyle{alpha}

%\bibliography{mybib7}

\section{Introduction}

\label{introducao}

\addcontentsline{toc}{section}{Introduction}

The Thermodynamic Formalism, inspired by Statistical Mechanics, is a part of Ergodic Theory  that studies the behavior of invariant measures with regular Jacobians called equilibrium states.
We formally define equilibrium states with respect to a potential as follows.

\begin{Def}\label{est eq}
Consider a continuous map $G: \Lambda \to \Lambda$ on a compact metric space $\Lambda$. We say 
that a $G$-invariant probability measure $\mu$ is an \em{equilibrium state} for $G$ w.r.t. a 
potential $\phi: \Lambda \to \mathbb{R}$ if it satisfies
          $$h_{\mu} (G) + \int \phi \, d\mu  = \sup_\eta \left\{ h_{\eta} (G) + \int \phi \, d\eta 
\right\},  $$
where the supremum is taken over all $G$-invariant probability measures.
\end{Def}

The theory was introduced by Sinai in the 70s, and since then it has been extensively studied 
by many authors who have made great advances in the area. In \cite {Sinai}, the author 
initiated the study of equilibrium states for Anosov diffeomorphisms which would be 
supplemented by Bowen and Ruelle. See \cite {Ruelle68}, \cite {Ruelle78} and \cite{Bowen71}.
They considered a system that consists of an uniformly expanding map and a H\"older 
continuous potential. For that system they proved that it has an unique equilibrium state. Sinai, 
Ruelle and Bowen's strategy was to use a finite Markov partition to build a conjugacy with a 
Markov shift of finite type. With some adjustments the proof extends to the case of uniformly 
hyperbolic maps.

%In the context of flows Franco has extend the tecnique of Bowen to prove that there exists an 
%unique equilibrium states associates to a Expansive flow with specification and continuous 
%potentials with certain regularity, see \cite{Franco}. In 2015, Climenhaga and Thompson 
%\cite{CT} proved that the conclusion remains true under the assumption of a weaker version of 
%specification and expansivity.

In the context of non-uniformly hyperbolic (in dimension bigger than one), we mention the important contributions made by 
Sarig (see \cite {Sarig99} and \cite {Sarig03}), Buzzi and Sarig \cite{Buzzi Sarig}, who studied countable Markov shifts, Buzzi, Paccaut and Schmitt \cite{Buzzi},  who studied piecewise expanding maps, Oliveira and Viana \cite{OV08}, and Varandas and Viana \cite{Varandas}, who considered non-uniformly expanding maps, among others.

More specifically, in \cite{OV08},  the authors proved existence and uniqueness of 
equilibrium states for a family of non-uniformly expanding maps associated with a class of 
H\"older continuous potentials  with small variation. In their work, they considered  finite and 
not generating Markov partitions. Some rectangles are uniformly expanding and others could 
have some contractive behavior. They used the notion of {\it hyperbolic times} introduced by 
Alves, Bonatti and Viana in \cite{ABV} to construct special Gibbs-type measures.

Leplaideur, Oliveira and Rios \cite {LOR}  studied the existence and properties of equilibrium states, for a family of partially hyperbolic 
horseshoes introduced by Diaz {\it et alii} in \cite{ diaz et al}. The three dimensional 
horseshoes were proved to be simultaneously at the boundary of the set of uniformly 
hyperbolic systems and the set of persistently non hyperbolic ones (see \cite{diaz et al}). In 
\cite {LOR}, the authors studied the ergodic measures with support in the limit set, showing 
that there is a spectral gap on the set of central Lyapunov exponents. As a consequence, there 
is a phase transition for the potential $t .log{|J^c|}$, where $J^c$ is the Jacobian of the system 
restricted to the (one dimensional) central direction and $t\geq 0$. In fact, for a special value $t_0$, there 
are two mutually singular equilibrium states.  For $t>t_0$, the equilibrium state is unique, and 
it is a Dirac measure. They proved the existence of equilibrium states for every $t$ (for every 
continuous potential, in fact), but the question of uniqueness for $0\leq t<t_0$  was left open. 
In \cite {katrin}, the authors explored geometric  and ergodic properties of a class of 
horseshoes with similar dynamical properties.

Still in the context of the partially hyperbolic horseshoes introduced in \cite{diaz et al}, Arbieto and Prudente \cite{AP} proved that, for H\"older continuous 
potentials depending only on the unstable direction (thus constant on the center-stable one), 
the associated equilibrium state is unique.  

Here we consider the same family of three dimensional horseshoes, taking the inverse map. We start by defining a two dimensional abstract map,  obtained by 
projecting the  system on two center-unstable leaves (center-stable leaves for the original map).  We prove existence and uniqueness of 
equilibrium states, associated to  H\"older continuous potentials with small variation, where, by small, we mean smaller than a fixed constant explicitly given. This part of the work is strongly inspired in the work of Oliveira and Viana  \cite {OV08}, but there are many extra difficulties to circumvent in our case: the degree of the map is not constant, and the number of expanding rectangles of the Markov partition is too small. 

Back to the three dimensional problem, we address the question  of uniqueness of equilibrium states for a special class of H\"older continuous potentials, that 
depend only on the center-unstable direction (complementing the direction considered in 
\cite{AP}), and whose variation has the same bound as before. 

This paper is organized as follows. In section~\ref{main theorems},  we introduce the family of horseshoes 
constructed in the paper \cite {diaz et al}  and the projected maps $ G=G_{\lambda,\sigma}$. We also state the main results of our work. The first main 
result  (Theorem \ref{Teorema Principal}) stablishes the class of potentials we deal with, and ensures the 
uniqueness of equilibrium states for the family $G$. Theorem \ref{Teorema B} has  intrinsic interest, and 
will be used for the proof of Theorem \ref{Teorema Principal}. It guarantees the existence of an eigenfunction 
and eigenmeasure for the Ruelle-Perron-Frobenius transfer operator and its dual, respectively, 
both associated with the same eigenvalue that is the spectral radius of the operator. Theorem~\ref{Teo C} gives an answer to the three dimensional problem for a family of potentials varying only on the center-unstable direction.

In section \ref{prova teo B} we prove Theorem \ref{Teorema B}.  We establish the existence of a reference measure 
and prove the abundance of hyperbolic times for almost every point with respect to it. We find 
an eigenfunction for the Transfer Operator and state some of its properties.  In sections \ref{Existence2}, and \ref{Uniqueness} 
we prove the existence and uniqueness of the equilibrium state, which is a non-lacunary Gibbs 
measure.  Finally, in section \ref{back}, we prove Theorem~\ref{Teo C}. We use the equilibrium state for the map $G$ to construct 
an equilibrium state supported on the three dimensional horseshoe, and prove the 
uniqueness.

%%%%%%%%%%%%%%%%%%%%%%%%%%%%%%%%%%%%%%%%%%%%%%%%%%%%%%%%%%%%%%%%%%%%%%%%%%%%%%%%%%%%%%%
%%%%%%%%%%%%%%%%%%%%%%%%%%%  SECTION %%%%%%%%%%%%%%%%%%%%%%%%%%%%%%%%%%%%%%%%%%%%%%%%%%
%%%%%%%%%%%%%%%%%%%%%%%%%%%%%%%%%%%%%%%%%%%%%%%%%%%%%%%%%%%%%%%%%%%%%%%%%%%%%%%%%%%%%%%%
%%%%%%%%%%%%%%%%%%%%%%%%%%%%%%%%%%%%%%%%%%%%%%%%%%%%%%%%%%%%%%%%%%%%%%%%%%%%%%%%%%%%%%%%

\section{Main results}\label{main theorems}

Consider $R= [0,1]\times[0,1]\times[0,1]\subset\mathbb{R}^3$ and let 

$$
\tilde{R}_0 =[0,1]\times [0,1]\times [0,1/6]
\qquad \mbox{and} \qquad
\tilde{R} _1=[0,1]\times [0,1]\times [5/6,1].
$$ 
Define 
   $$ F_{0}(x,y,z):=F_{| \tilde{R}_{0} }(x,y,z) =(\rho x , f(y),\beta z),$$
where $0 < \rho <{1/3}$, $\beta> 6$ and  $f(y) =\frac {1}{1 - (1-{1/y})e^{-1}}. $

\begin{figure}[h!]
\centering
\includegraphics[scale=0.3]{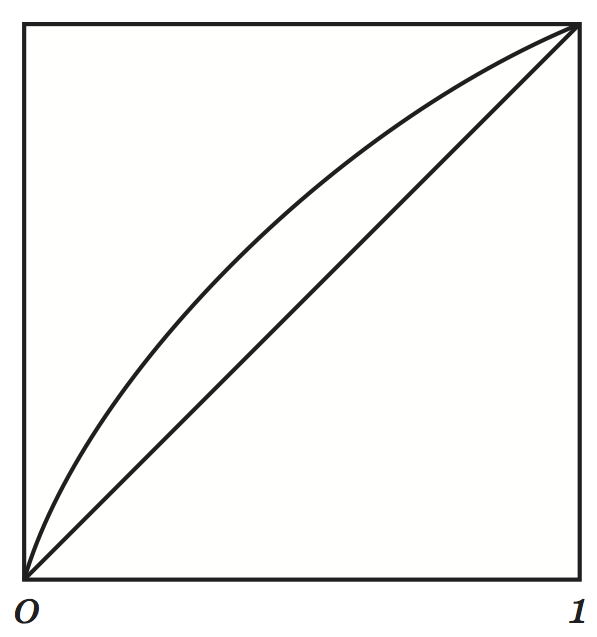}
\caption{Function $f$.}\label{tempoum}
\end{figure}

We also define
$$F_{1}(x,y,z)=F_{| \tilde{R}_{1} }(x,y,z) = \Big(\frac{3}{4}- \rho x , \sigma (1 - y) ,\beta_{1} \Big(z - \frac{5}{6} \Big)\Big),$$
where  $0<\sigma< {1/3}$ e $3< \beta_1 < 4$.

For $X \in R$, we define 
\begin{equation}
\label{Fdeles}
F(X) = \left\{
\begin{array}{lcl}
F_0(X) & if & X \in \tilde R_0\\
F_1(X) & if & X \in \tilde R_1.
\end{array}
\right.
\end{equation}

Points $X \in R$ outside $\tilde{ R}_0 \cup \tilde {R}_1$ will be mapped  injectively outside $R$. Notice  that $F$ depends on the parameters $\rho, \beta, \beta_1$ and $\sigma$.

\begin{figure}[h!]
\centering
\includegraphics[scale=0.5]{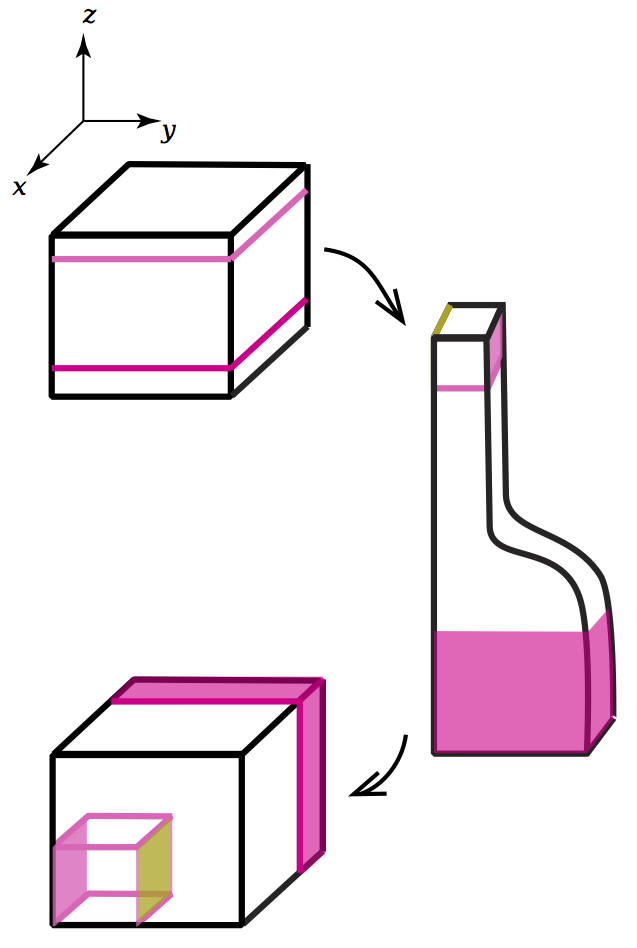}
\caption{The map $F$}\label{ferradura2}
\end{figure}

In \cite{diaz et al}, Díaz et al showed that the non-wandering set for $F$ is partially hyperbolic, with one dimensional central direction, parallel to the $y$-axis. The vertical direction is expanding and the horizontal direction, parallel to the $x$-axis is contracting. The $y$-axis direction has both behaviours.

Now we define the map $G$, which is related to the projection of $F^{-1}$ in two center-stable planes.
We consider an abstract space consisting by  three rectangles  $R_1$, $R_2$ and $R_3$  defined as follows:
\[
\begin{split}
& R_1 = [0,\rho]\times [0,1]\times \{0\}, \\
& R_2 = [3/4 - \rho,3/4]\times [0,\sigma]\times \{0\}, \\
& R_3 = [0,\rho]\times [2,3]\times \{5/6 \}.
\end{split}
\]

\begin{figure}[h!]
\centering
\includegraphics[scale=0.3]{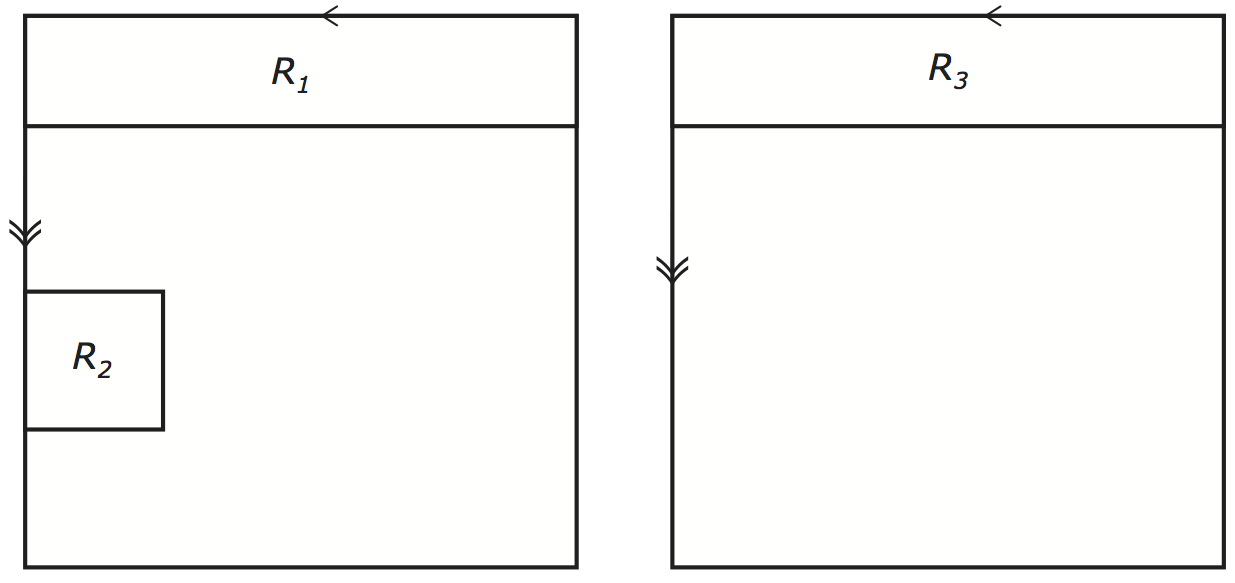}
\caption{Rectangles $R_1$, $R_2$ and $R_3$.}\label{retangulos}
\end{figure}

The rectangles are inside two planes called $P_0$ and $P_1$ as in figure~\ref{retangulos}. 
Let  $g_0: [0,1] \to \mathbb{R}$, $g_0(y)=f^{-1}(y) $ and let $g_1:[0,1] \to \mathbb{R}$ be defined by $g_1(y)= 1- \sigma^{-1}y$. See their graphs in figure~\ref{SIF}.
\begin{figure}[h!]
\centering
\includegraphics[scale=0.3]{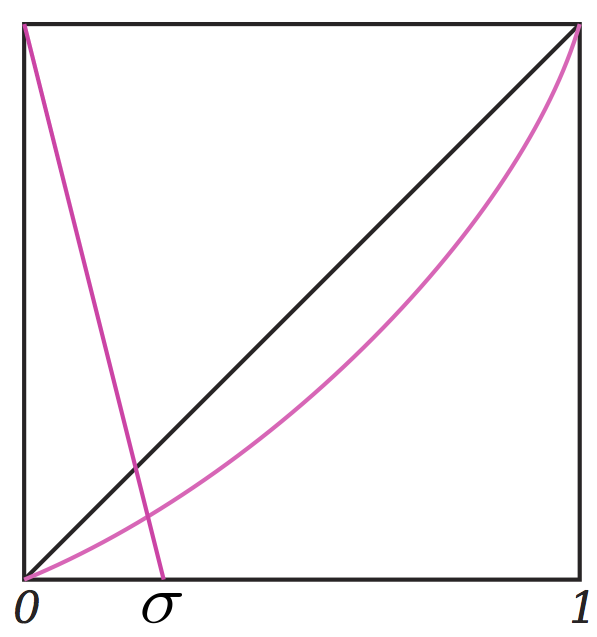}
\caption{Functions $g_0$ and $g_1$.}\label{SIF}
\end{figure}

Let $\alpha = \rho^{-1}$. We define a map $G: \displaystyle\bigcup_{i=1}^{3}{R_{i}} \to  P_{0} \cup  P_{1}$ by its restrictions $G_i$ to each rectangle $R_i$ as follows:
\begin{eqnarray*}
  G_{1}(x,y,z) & = & \big(\alpha x, g_{0}(y),0  \big), \\        
  G_{2}(x,y,z) & = & \big(\alpha (3/4 - x),g_{1}(y) ,5/6  \big), \\ 
  G_{3}(x,y,z) & = & \big(\alpha x,g_{0}(y),0  \big). 
\end{eqnarray*}   

The behaviour of the map $G$ is similar on $R_1$ and $R_3$. In these rectangles, the vertical direction is expanding. In the horizontal direction $G$ sends the points from the right side to the left, except for the extreme points whose $x$ coordinates are fixed. In the rectangle $R_2$ both directions are expanding.

We consider $\Lambda$ the maximal invariant set of the union of the rectangles: 
\begin{equation*}
\label{maxinv}
\Lambda := \displaystyle\bigcap _{n \in \mathbb{N}}G^{-n}  \Big(\displaystyle\bigcup _{i=1}^{3}R_i  \Big),
\end{equation*}

and from now on we call $G$ the restriction $G|_{\Lambda}: \Lambda \rightarrow \Lambda$.

\begin{figure}[h!]
\centering
\includegraphics[scale=0.3]{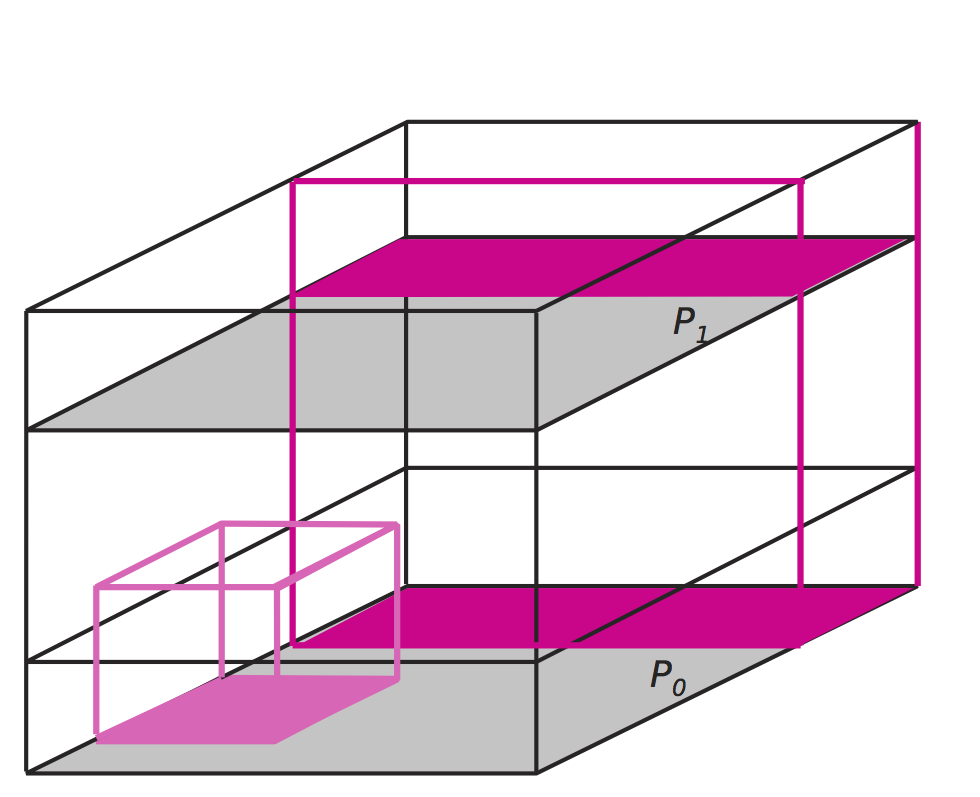}
\caption{Horizontal planes}\label{ferradura2}
\end{figure}

Let $A$ be the matrix: 
$$ A = \left( 
\begin{array}{ccc}
   1 & 1 & 0 \\
   0 & 0 & 1 \\
   1 & 1 & 0 \\      
\end{array}
\right).  $$ 
and consider $\Sigma _{A}$ the finite type subshift with transition matrix $A$:  
$$\Sigma_{A}= \left\{ \Theta =  \left( \theta _{0} \theta _{1} \theta _{2} \cdots \right) \in \left\{1,2,3 \right\} ^ {\mathbb{N}} | A_{\theta_{i}\,\theta_{i+1}} = 1 \right\}. $$

According to this the following transitions are allowed:
\begin{eqnarray*}
   & &   1 \rightarrow 1, 2 \\
   & &   2 \rightarrow 3\\
   & &   3 \rightarrow 1, 2 .\\
\end{eqnarray*} 
The map $G$ is semi-conjugated to the shift $\sigma$ on $\Sigma_A$.  This result is an adaptation of a result in \cite{diaz et al}. [See Lemma 5.2]  

Note that the points belonging to the rectangles 1 and 2 have two pre-images while points in the rectangle 3 have just one pre-image.

\begin{figure}[h] \centering \label{iterados}
\begin{minipage}[b]{0.45\linewidth}
\includegraphics[width=\linewidth]{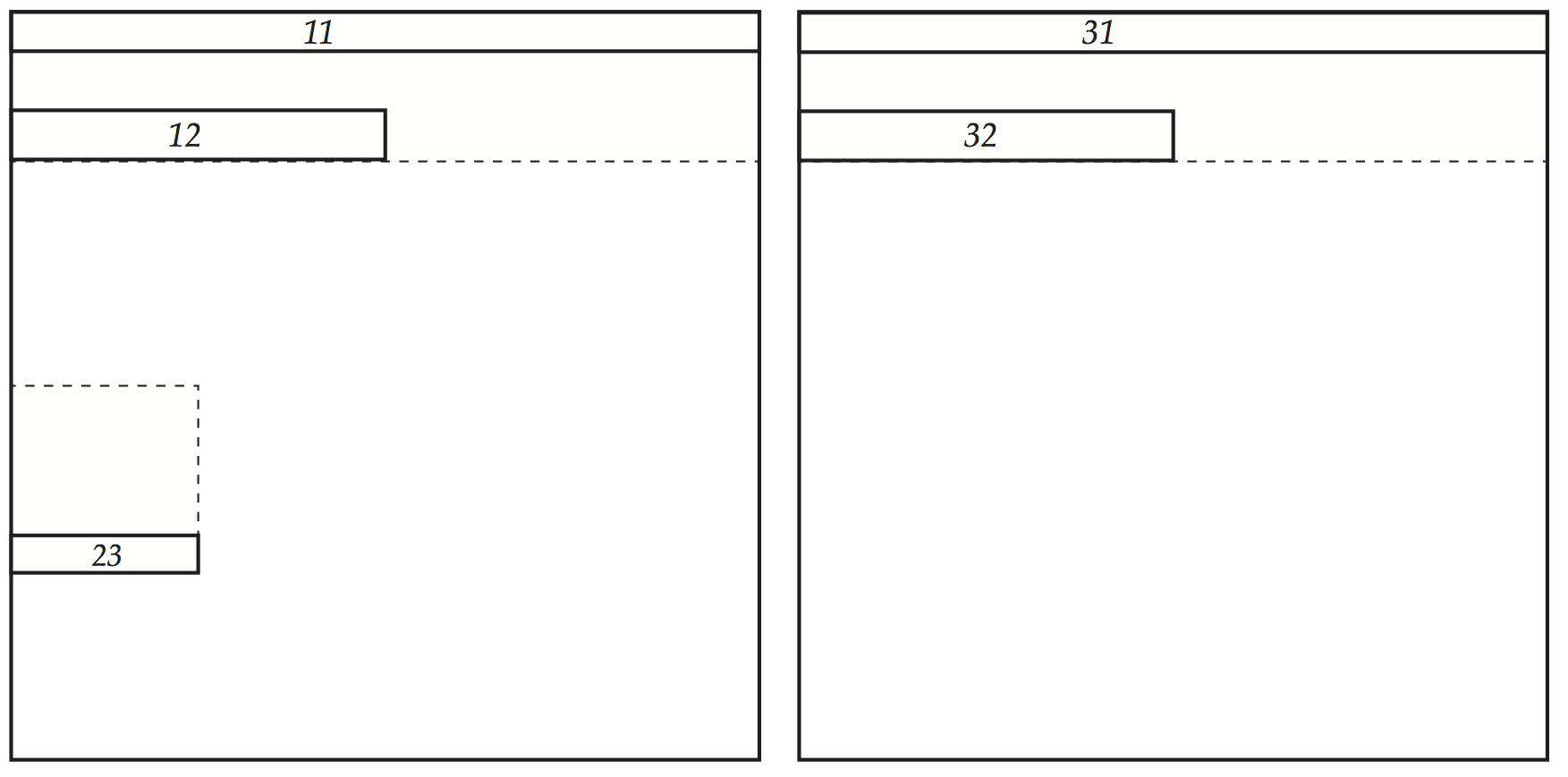} \caption{Second generation}
\end{minipage} 
\hfill
\begin{minipage}[b]{0.45\linewidth} 
\includegraphics[width=\linewidth]{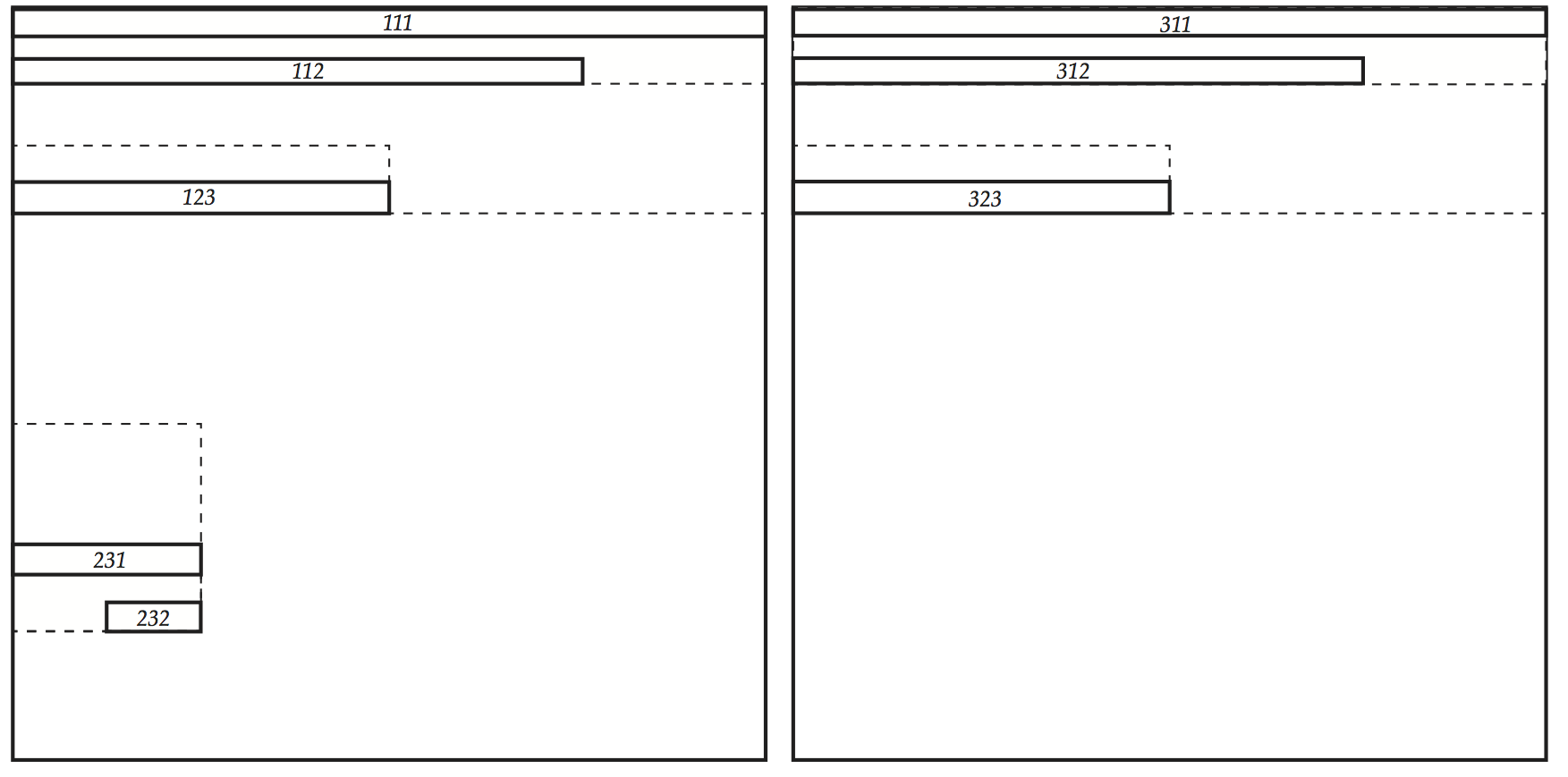} \caption{Third generation} 
\end{minipage}
\end{figure}

In figure~\ref{iterados}, the steps in the generation of the set $\Lambda$ are depicted. Note that $\Lambda$ is not a Cantor set: it contains infinitely many lines segments.
It is worth to evidence that $G$ is not conjugated to the shift $\sigma$. The entire segment $[0,1]$ is associated to a unique sequence on $\Sigma_A$: the constant sequence $(\cdots 1,1,1 \cdots)$.

The topological entropy of the subshift $\sigma$ is given by:
 $$h_{top}(\sigma)=\log \left(\frac{1+\sqrt{5}} {2}\right).$$ 
By the semiconjugacy we have $h_{top}(G) \geq \log \left(\frac{1+\sqrt{5}} {2}\right)$. Let  $\omega=\left(\frac{1+\sqrt{5}} {2}\right)$.
We are ready to formulate the first main result of this paper.

\begin{maintheorem}\label{Teorema Principal}
Let $G \colon \Lambda \to \Lambda$ be as above and $\phi : \Lambda \rightarrow \mathbb{R}$ be a H\"older continuous potential satisfying:
\begin{equation}\label{variacaopequena}
\sup \phi - \inf \phi < \frac{\log \omega}{2}.
\end{equation}
Then there exists a unique equilibrium state for the system $G$ with respect to the potential $\phi$.
\end{maintheorem}
We point out that our class of potentials includes the zero potential, constant potentials and potentials similar to the ones studied in \cite{LOR} with $t$ close to zero.

%%%%%%%%%%%%%%%%%%%%%%%%
%%%%%%%%%%%%%%%%%%%%%%%%%
%%%%%%%%%%%%%%%%%%%%%%%%%%%

%%%%%%%%%%%%%%%%%%%%%%%%%%%%%%%%%%%%%%%%%%%%%%%%%%%%%%%%%%%%%%%%%%%%%%%%%%%%%%%%%
%%%%%%%%%%%%%%%%%%%%%%%%%%%%%%%%%%%%%%%%%%%%%%%%%%%%%%%%%%%%%%%%%%%%%%%%%%%%%%%%%
%%%%%%%%%%%%%%%%%%%%%%%% OPERADOR DE TRANSFERENCIA  %%%%%%%%%%%%%%%%%%%%%%%%%%%%
%%%%%%%%%%%%%%%%%%%%%%%%%%%%%%%%%%%%%%%%%%%%%%%%%%%%%%%%%%%%%%%%%%%%%%%%%%%%%%%%%
%%%%%%%%%%%%%%%%%%%%%%%%%%%%%%%%%%%%%%%%%%%%%%%%%%%%%%%%%%%%%%%%%%%%%%%%%%%%%%%%%

\subsection{ Ruelle-Perron-Fr\"obenius Operator}
\label{operador transf}

Let $\mathcal{C}(\Lambda)$ be the set of  real continuous functions on $\Lambda$. For a H\"older continuous potential $\phi$ we define an operator $\mathcal{L}_{\phi} : \mathcal{C} \left( \Lambda \right) \rightarrow \mathcal{C}\left(\Lambda \right)$,  called the {\emph{ Ruelle-Perron-Fr\"obenius operator}} which associates for each $\psi \in \mathcal{C} \left( \Lambda \right)$ a function
$ \mathcal{L}_{\phi} (\psi) \colon \Lambda \to \mathbb{R}$ by:
\begin{equation}
\label{optransf}
\mathcal{L}_{\phi} \psi \left(x\right) = \displaystyle\sum _{y  \in \, G^{-1}\left(x\right)} e^{\phi \left(y \right)} \psi \left( y \right). 
\end{equation}
 For each $n \! \in\! \mathbb{N}$ we consider the Birkhoff sum $S_{n}\phi(x)= \displaystyle\sum_{j=0}^{n-1} \phi\big(G^{j}(x)\big)$. Note that, iterating the operator we obtain,  for each $n \in \mathbb{N}$,  
\begin{equation}
\label{iteradostransf}
\mathcal{L}_{\phi}^{n} \psi \left(x\right) = \displaystyle\sum _{y  \in \, G^{-n}\left(x\right)} e^{S_{n}\phi \left(y \right)} \psi \left( y \right).
\end{equation} 
We also consider the dual operator of $ \mathcal{L}_{\phi}$, $ \mathcal{L}_{\phi}^{\ast}: \mathcal{M}(\Lambda) \to \mathcal{M}(\Lambda)$ that satisfies
$$\int \psi \ d\mathcal{L}_{\phi}^{\ast}\eta  = \int  \mathcal{L}_{\phi}(  \psi ) \ d\eta , $$
for every $\psi \in  \mathcal{C}(\Lambda) $ and every $\eta \in \mathcal{M}(\Lambda)  $.
Now we can stablish the second main theorem of this paper. Consider $\phi$ as in the statement of the Theorem~\ref{Teorema Principal}.

\begin{maintheorem}\label{Teorema B}
Let $\lambda$ be the spectral radius of the transfer operator $\mathcal{L}_{\phi}$.
There exists a probability measure $\nu \in \, \mathcal{M} \left(\Lambda\right)$ and  there exists a positive function $h: \Lambda \to \mathbb{R}$ wich satisfies:
$$\mathcal{L}_{\phi}^{\ast} \nu = \lambda \nu  \quad \mbox{e} \quad  \mathcal{L}_{\phi} h = \lambda h .  $$  
\end{maintheorem}

\begin{maintheorem}\label{Teo C}
Let $F: \tilde{R}_0 \cup \tilde{R}_1 \to R$ be the three dimensional partially hyperbolic horseshoe defined in the begin of this section. Let $ \varphi: \tilde{R}_0 \cup \tilde{R}_1 \to \mathbb{R}$ be a H\"older continuous potential with $\sup \varphi - \inf \varphi < \log\frac{\omega}{2}$. Assume that $\varphi$ does not depend on the $z$-coordinate in each set$\tilde{R}_0$ and $\tilde{R}_1 $. Then there exists a unique equilibrium state for the system $F$ with respect to the potential $\varphi$.
\end{maintheorem}

%
%Besides we use the Theorem~\ref{Teorema B} to prove the Theorem~\ref{Teorema Principal} it has its own interested. 

%%%%%%%%%%%%%%%%%%%%%%%%%%%%%%%%%%%%%%%%%%%%%%%
%%%%%%%%%%%%%%%%%%%%%%%%%%%%%%%%%%%%%%%%%%%%%%%
%%%%%%%%%%%%%%%%%%%%%%%%%%%%%%%%%%%%%%%%%%%%%%%

%%%%%%%%%%%%%%%%%%%%%%%%%%%%%%%%%%%%%%%%%%%%%%%%%%%%%%%%%%%%%%%%%%%%%%%%%%%%%%%%%%%%%%%%
%%%%%%%%%%%%%%%%%%%%%%%%%%%  SECTION %%%%%%%%%%%%%%%%%%%%%%%%%%%%%%%%%%%%%%%%%%%%%%%%%%
%%%%%%%%%%%%%%%%%%%%%%%%%%%%%%%%%%%%%%%%%%%%%%%%%%%%%%%%%%%%%%%%%%%%%%%%%%%%%%%%%%%%%%%%
%%%%%%%%%%%%%%%%%%%%%%%%%%%%%%%%%%%%%%%%%%%%%%%%%%%%%%%%%%%%%%%%%%%%%%%%%%%%%%%%%%%%%%%%

\section{ Proof of Theorem B }
\label{prova teo B}

We consider $K\!=\!\{ f  \in \mathcal{C}(\Lambda) \, | \, f \geq 0 \}$ the cone of continuous non-negative functions on $\Lambda$. Then $K$ is a normal cone which gives a partial order on $\mathcal{C}(\Lambda)$. Since $\mathcal{L}_{\phi}$ is a positive and bounded operator, it follows from Mazur's Theorem [see \cite{analise funcional 2}] that:

\begin{Prop}\label{med referencia}
The spectral radius of $\mathcal{L}_{\phi}$, $\lambda$, is an eigenvalue for the dual operator $\mathcal{L}_{\phi}^ {\ast} $ associated an eigenmeasure $\nu$.
\end{Prop}

We can stablish a lower bound to the eigenvalue of $\mathcal{L}_{\phi}^{\ast}$.

\begin{Lema}\label{cota lambda}
The spectral radius of  $\mathcal{L}_{\phi}$ , $r(\mathcal{L}_{\phi})= \lambda$, sastisfies
              $$\lambda \geq e^{\inf \phi + \log \omega} ,   \qquad \mbox{where} \ \omega = \left(\frac{ 1 + \sqrt{5}}{2}\right).$$ 
\end{Lema}
\begin{proof}
By definition $\lambda = r(\mathcal{L}_{\phi})= \limsup \| \mathcal{L}_{\phi}^{n}  \| ^{\frac{1}{n}}.$ Therefore:
$$\| \mathcal{L}_{\phi}^{n} \mathtt{1} (x) \| = \displaystyle\sum_{y \in G^{-n}(x)} e^{S_{n}\phi (y)} \geq \# \{G^{-n}(x) \}\cdot e^{n \inf \phi}.$$ 
Then, we have
   $$\lambda \geq \limsup [\# \{G^{-n}(x) \}\cdot e^{n \inf \phi}]^{\frac{1}{n}} = e^{\inf \phi } e^{\log \omega} = e^{\inf \phi + \log \omega}$$
since $ \limsup_{n \rightarrow \infty} \frac{1}{n} \log \# \{G^{-n}(x) \} =\log \omega $.
\end{proof}

%%%%%%%%%%%%%%%%%%%%%%%%%%%%%%%%%%%%%%%%%%%%%%%%%%%%%%%%%%%%%%%%%%%%%%%%%%%%%%%%%%%%%%%%%%%%%%%%%%%%%%%%%%%%%%%%
%%%%%%%%%%%%%%%%%%%%%%%%%%%%%%%%%%%%%%%%%%%%%%%%%%%%%%%%%%%%%%%%%%%%%%%%%%%%%%%%%%%%%%%%%%%%%%%%%%%%%%%%%%%%%%%%
%%%%%%%%%%%%%%%%%%%%%%%%%%%%% SECTION %%%%%%%%%%%%%%%%%%%%%%%%%%%%%%%%%%%%%%%%%%%
%%%%%%%%%%%%%%%%%%%%%%%%%%%%%%%%%%%%%%%%%%%%%%%%%%%%%%%%%%%%%%%%%%%%%%%%%%%%%%%%%%%%%%%%%%%%%%%%%%%%%%%%%%%%%%%%
%%%%%%%%%%%%%%%%%%%%%%%%%%%%%%%%%%%%%%%%%%%%%%%%%%%%%%%%%%%%%%%%%%%%%%%%%%%%%%%

\subsection{Hyperbolic times}
\label{Hyperbolic times}

Given $n\in \mathbb{N}$ and $n$ symbols $a_{j} \in \{ 1,2,3 \}$ we define the \emph{cylinder set of length} $n$ by:
$$\mathcal{R}^{n}\left( a_{0}, \cdots , a_{n-1} \right)= \left\{ x \! \in \! \Lambda \, | \, x \! \in \! R_{a_{0}},G(x) \! \in \! R_{a_{1}}, \cdots,G^{n-1}(x) \! \in \! R_{a_{n-1}}  \right\}.$$ 
When there is no risk of confusion, we  denote $\mathcal{R}^{n}$ to refer to $ \mathcal{R}^{n}\left( a_{0}, \cdots , a_{n-1} \right)$.

For each $\gamma \in (0,1)$ and $n \geq 1$  let  $I(\gamma,n)$ be the set of  admissible words of length $n$ such that the proportion of symbols $1$ and $3$ are at least $\gamma$. That is 
  $$ \# \left\{0 \leq j \leq n -1 | a _{j} = 1,3 \right\} > \gamma n . $$ 
%If $\gamma$ satisfies
%\begin{equation}\label{condicao sobre gamma}
%   \sigma ^{ - \left(1 - \gamma \right)} \cdot \left(\frac{1}{e}\right)^{\gamma} > 1 ,      
%\end{equation}
%there is $c>0$  such that
%\begin{equation}\label{condicao sobre c}
%   \sigma ^{ - \left(1 - \gamma \right)} \cdot \left(\frac{1}{e}\right)^{\gamma} > e^{ 4 c} > 1 .      
%\end{equation}
%
%For that it is enough to consider $ 0< \gamma < \frac{\log(\sigma^{-1})}{\log(\sigma^{-1} + 1)}$ . 
 
We can associate words in $I(\gamma,n)$ with orbits, or more precisely, with cylinders sets.  The points in these cylinders spend a big fraction of their orbits in the "non-hyperbolic" part of the system.
We will  study the  exponential growth of the cardinality of these sets $I(\gamma, n)$. 
We can estimate the cardinality by counting the number of admissible words of length $n$ whithin at most $(1- \gamma) n$ symbols $2$. 

By definition, to compose a word with $k$ symbols $2$ we have to follow the conditions below:
\begin{itemize}
\item a symbol two is followed by a three.
\item a symbol $3$ is followed by $1$ or $2$.
\item a symbol $1$ is followed by $1$ or $2$.
\end{itemize}

It means that once we have distributed the symbols $2$ the other symbols, except possibly for the first one, are already determined. Following this rules the number of words of length $n$ with $k$ symbols $2$ is given by $\left( 
           \begin{array}{c}                                                                        
           n-k \\
           k \\
           \end{array}  
           \right). $
If the first symbol in the word is not a $2$ we have two choices: $1$ or $3$.  According to this we have the follow estimate:
%
%Para distribuir os símbolos $2$ temos $
%\left(
%\begin{array}{c}                                                                        
%           n-k \\
%           k \\
%           \end{array}  
%\right)           $
%escolhas, visto que cada $2$ vem seguido de um símbolo 3.         
\begin{equation}
           \# I(\gamma,n) \leq \displaystyle\sum _{k=0}^{(1-\gamma)n}2\left( 
           \begin{array}{c}                                                                        
           n-k \\
           k \\
           \end{array}  
           \right).  \label{cardinalidade}        
\end{equation}

%\begin{Lema}[Fórmula de Stirling]\label{formula de stirling}
%Dado $n\geq 1$ tem-se
% $$ n! = \sqrt{2 \pi n} \left( \frac{n}{e}\right)^{n}e^{\lambda_{n}} \qquad \mbox{onde} \qquad   \frac{1}{12n + 1} < \lambda_{n} <  \frac{1}{12n} .$$
%\end{Lema}

\begin{Lema}\label{cota fatorial}
For each $n \in \mathbb{N}$, and $0 < k < n$ we have: 
\begin{eqnarray*}
           \left( 
           \begin{array}{c}
           n-k \\
           k \\
           \end{array}  
           \right)
&\leq & \frac{e^k}{\sqrt{2\pi k}\cdot e^{\frac{1}{12k + 1}}}\cdot \left(\frac{n}{k} -1 \right)^k.
\end{eqnarray*}
\end{Lema}
\begin{proof}
In fact,
\begin{eqnarray*}
           \left( 
           \begin{array}{c}
           n-k \\
           k \\
           \end{array}  
           \right)
&  =  & \frac{(n-k)!}{k!(n-2k)} = \frac{(n-k) (n-k-1)\cdots (n-k-(k+1))}{k!} \\
&\leq & \frac{(n-k)^{k}}{k!}.
\end{eqnarray*}
By Stirling's Formula, see \cite{Bartle},  we have:
$$k!= \sqrt{2 \pi k} \left( \frac{k}{e}\right)^{k}e^{\lambda_{k}} > \sqrt{2 \pi k} \left( \frac{k}{e}\right)^{k}e^{\frac{1}{12k + 1}}.$$
Then
\begin{eqnarray*}
           \left( 
           \begin{array}{c}
           n-k \\
           k \\
           \end{array}  
           \right)
&\leq &  \frac{(n-k)^{k}}{ \sqrt{2 \pi k} \left( \frac{k}{e}\right)^{k}e^{\frac{1}{12k+1}}}\\
& =   & \frac{e^k}{\sqrt{2\pi k}\cdot e^{\frac{1}{12k + 1}}}\cdot \left(\frac{n}{k} -1 \right)^k,    
\end{eqnarray*} 
which finishes the proof of the lemma.
\end{proof}

\begin{Lema}\label{decrescente} 
For $n$ large enough and $0 <\alpha < \frac{\sqrt{2}-1}{2}$ we have: 
 $$\left( 
           \begin{array}{c}
           n-k \\
           k \\
           \end{array}  
           \right)
            < \left( 
           \begin{array}{c}
           n-(k+1) \\
           k + 1 \\
           \end{array}  
           \right) \mbox{for all } k\leq [\alpha n].$$ 
\end{Lema}

\begin{proof}           
In fact, we have that the inequality 

\begin{eqnarray*}
& &\left( 
           \begin{array}{c}
           n-k \\
           k \\
           \end{array}  
           \right) = \frac{(n-k)!}{k!\cdot (n - 2k)!} < \frac{(n - (k+1))!}{(k+1)! (n - 2(k+1))!}=\\ &=&\left( 
           \begin{array}{c}
           n-(k+1) \\
           k + 1 \\
           \end{array}  
           \right)\\
 \end{eqnarray*}          
holds, if and only if:
       $$ \frac{(n-k) \cdot (k+1)}{(n- 2k)\cdot (n - 2k -1)} < 1. $$
If we have $n^2 -4nk - n + 4k^2 +2k > 0$, then the previous inequality is equivalent to the following:
      $$ -5k^{2}+ (5n -3)k + (- n^{2}+ 2n) < 0.$$  
Which is always true.
Then it is enough to verify conditions to give us $n^2 -4nk - n + 4k^2 +2k > 0$. 
If we make $k \leq [\alpha n]$ we have:
$$n^2 - 4kn \geq n^2 -4\alpha n^2 -n > 4\alpha^2n^2 + 2\alpha n \geq 4k^2+2k ,$$
where the first and the last inequalities always hold.To get the second inequality, it is enough to prove that:
$$n^2(-4\alpha - 4 \alpha ^2 +1) - n(1 -2\alpha) > 0.$$
Taking $0< \alpha < \frac{\sqrt{2}-1}{2}$, the coefficient of $n^2$ is positive, and the inequality is true for $n$ large enough.
\end{proof}

Let  $0< \alpha < \frac{\sqrt{2}-1}{2}$ and consider  $\alpha=\alpha(\gamma) = 1- \gamma$. Then 
\begin{eqnarray*}
\# I(\gamma,n) \leq 2 [\alpha n] 
        \left( 
           \begin{array}{c}
           n - [ \alpha n ]  \\
         { [ \alpha n ] }  \\
           \end{array}  
           \right).
\end{eqnarray*}
Applying the lemma \ref{cota fatorial} we obtain:
\begin{eqnarray*}
\# I(\gamma,n)    & \leq & \alpha n \cdot \frac{e^{\alpha n}}{\sqrt{2\pi \alpha n}\cdot e^{\frac{1}{12\alpha n + 1}}}\cdot \left(\frac{n}{\alpha n} -1 \right)^{\alpha n} \\
& \leq & \alpha n \cdot \frac{e^{\alpha n}}{\sqrt{2\pi \alpha n}\cdot e^{\frac{1}{12\alpha n + 1}}}\cdot \left[ \left(\frac{1 - \alpha}{\alpha   }\right)^{\alpha}\right]^{n}.  
\end{eqnarray*}
%Estudando o crescimento exponencial, temos : 
%\begin{eqnarray*}
%\frac{1}{n} \log (\# I(\gamma,n) ) \leq \frac{1}{n} \log \left( \alpha n \cdot \frac{e^{\alpha n}}{\sqrt{2\pi \alpha n}\cdot e^{\frac{1}{12\alpha n + 1}}}\cdot \left[ \left(\frac{1 - \alpha}{\alpha   }\right)^{\alpha}\right]^{n} \right).
%\end{eqnarray*} 

Note that, if we make $\gamma \rightarrow 1$, we have $\alpha \rightarrow 0$ and hence 
   $\left(\frac{1 - \alpha}{\alpha   }\right)^{\alpha}$ converges to one. Then we can establish the following proposition.

\begin{Prop}\label{cilindros ruins} 
If $\beta =\frac{ \log \omega}{2}$ then there exists $0< \gamma < 1$  which satisfies:

      $$ \# I(\gamma , n) \leq e^{\beta n} \quad \mbox{for} \ n \ \mbox{large enough}. $$
\end{Prop}

\begin{proof}
We consider
\begin{eqnarray*}
f\left( \alpha ,n\right)  &:=&\frac{1}{n}\ln \left( \alpha n\frac{e^{\alpha n}%
}{\sqrt{2\pi \alpha n}e^{\frac{1}{12\alpha n+1}}}\left( \frac{1-\alpha }{%
\alpha }\right) ^{\alpha n}\right) = \\
&=&\frac{1}{n}\left( \frac{\ln \alpha +\ln n}{2}-\frac{1}{12\alpha n+1}-%
\frac{\ln \left( 2\pi \right) }{2}\right) +\\
& +&\alpha +\alpha \ln \left(
1-\alpha \right) -\alpha \ln \alpha.  
 \end{eqnarray*}
Taking the limit in $n$ we have
$$h\left( \alpha
\right):=\lim_{n\rightarrow \infty }f\left( \alpha ,n\right)  =\alpha +\alpha \ln \left( 1-\alpha \right) -\alpha \ln \alpha
=\alpha +\alpha \ln \left( \frac{1-\alpha }{\alpha }\right).$$
Since $\displaystyle\lim_{\alpha \to 0} h(\alpha)= 0$  and $h(\frac{\sqrt{2}-1}{2}) > 0.48$, there exists $\alpha_{0} \in \big(0, \frac{\sqrt{2}-1}{2} \big)$  with $h(\alpha_{0}) = \frac{\log(\omega)}{4}$ since $\frac{\log \omega}{4} < 0.3$.
We have that $f(\alpha, n)$ converges to $h(\alpha)$ uniformly on any compact subset of $]0,1[$. Then, for $n$ sufficiently large, 
 $|f( \alpha_{0},n) - h(\alpha_{0})| < \frac{\log \omega}{4}.$ 

Thus $f(\alpha_{0},n)< \frac{ \log \omega}{2}$ and, therefore, there exists $\beta= \beta(\alpha_{0})$ satisfying the lemma. 
\end{proof}

Now we define a key tool in this paper: the concept of \emph{hyperbolic times}, which was introduced by Alves in \cite{Alves} and extensively studied in \cite{ABV} by Alves, Bonatti and Viana.

\begin{Def} 
We say that $n \in \mathbb{N}$ is a (G,c)-\emph{hyperbolic time}  for $x \! \in \! \Lambda$, for some $c>0$ if
$$\displaystyle\prod_{j=k}^{n-1} \left\| DG\left(G^{j}\left(x\right)\right) ^{-1}\right\|^{-1} \geq e^{2c \left(n-k\right)}, \quad    0 \leq k \leq n-1. $$
\end{Def}

%Vamos mostrar que em nosso caso temos muitos tempos hiperbólicos. Para isso, iremos utilizar como ferramenta essencial o Lema de Pliss, enunciado abaixo:  
% 
%\begin{Lema}{[Lema De Pliss]}
%
%Sejam $A\geq c_{2}> c_{1}, \quad \theta = \frac{c_{2} - c_{1}}{A- c_{1}}$ e  $a_{1}, \cdots , a_{n}$  tais que
%$$    a_{1} + \cdots + a_{n} \geq c_{2}n \quad \mbox{e} \quad a_{s}\geq A \quad \forall \ s=1, \cdots, n.$$
%Então existem inteiros $l \geq \theta n$ e  $1 \leq n_{1} < n_{2} < \cdots  <n_{l}\leq n$ satisfazendo:
%$$a_{k}+ \cdots + a_{n_{i}} \geq c_{1}\left(n_{i}- k \right), \quad  \forall \ \mbox{e} \ 0\leq k\leq n_{i}-1 \quad  \forall \ i=1, \cdots ,l. $$
%\end{Lema}
%Para ver uma prova deste lema consulte \cite{ABV} ou \cite{Mane}.    
%A seguinte proposição  mostra a existência de pontos que possuem muitos tempos hiperbólicos. Esses pontos pertencem a ''conjuntos bons'' que denominaremos cilindros hiperbólicos. Formalmente, diremos que um cilindro $\mathcal{R}^n$ de tamanho $n$ é um cilindro hiperbólico quando para todo $x \in \mathcal{R}^n $ temos que $n$ é tempo hiperbólico de $x$. 

If $\gamma$ satisfies
\begin{equation}\label{condicao sobre gamma}
   \sigma ^{ - \left(1 - \gamma \right)} \cdot \left(\frac{1}{e}\right)^{\gamma} > 1 ,      
\end{equation}
there is $c>0$  such that
\begin{equation}\label{condicao sobre c}
   \sigma ^{ - \left(1 - \gamma \right)} \cdot \left(\frac{1}{e}\right)^{\gamma} > e^{ 4 c} > 1 .      
\end{equation}

For that it is enough to consider $ 0< \gamma < \frac{\log(\sigma^{-1})}{\log(\sigma^{-1} + 1)}$ .

From now on we will fix $c>0$ as in equation~\ref{condicao sobre c} and we will omit  $c$ and  $G$ in the notation, referring  to a hyperbolic time, instead of $(G,c)$-hyperbolic time.

In our setting we can prove that good admissible words are associated to cylinder sets with positive proportion of hyperbolic times. For this, we use the well known Pliss Lemma (see \cite{ABV}).

\begin{Prop}\label{temp hiperbolicos}
There exists $\theta >0 $ such that  for every $ n \geq 1$ and every admissible word $\left(a_{0}, \cdots, a_{n-1}\right) \notin I\left(\gamma, n\right)$ there is $l> \theta n $ and numbers $1 \leq n_{1} < \cdots < n_{l}\leq n$ such that  $R^{n_{j}}\left[i_{0},\cdots , i_{n_{j}-1}\right]$ is hyperbolic for for all $ j=1, \cdots, l$.
\end{Prop}
\begin{proof}
Given  $ i= 1,2,3$, consider 
    $ b_{i} = inf\left\{\log \left\|DG\left(x\right)^{-1}\right\|^{-1}| x \in R_{i}\right\}$.
Since we have in each rectangle the following: 
\begin{eqnarray*}
 \left\|DG\left(x\right)^{-1}\right\| & = & \sigma \quad \mbox{if} \  x \in R_{2} \\
 \left\|DG\left(x\right)^{-1}\right\| & \leq & e  \quad  \mbox{if} \  x \in R_{1} \ \mbox{ou} \  R _{3} ,   
\end{eqnarray*}  
we can estimate $b_i$ by
$$
b_{i} \geq \left\{
\begin{array}{rcl}
\log \left(\frac{1}{e}\right)= -1 & \mbox{if} & i=1,3\\
\log\left(\sigma^{-1}\right)      & \mbox{if} & i=2 . 
\end{array}
\right.
$$
Defining  $a_{t} = b_{i_{t}}$ where $t=0, \cdots, n-1$, and since  $\left(i_{0}, \cdots, i_{n-1}\right)$ belongs to the complement of $I\left(\gamma, n\right)$, we have
\begin{eqnarray*}
a _{1}+ \cdots + a_{n-1} & = &\displaystyle\sum_{t=0}^{n-1} b_{i_{t}} \geq \left(1 - \gamma \right)n \log \left(\sigma^{-1}\right) + \gamma n \left(-1\right)\\
                                      & > & 4cn
\end{eqnarray*}
where the last inequality follows from the choice of $c$ (see equation~\ref{condicao sobre c}).

Let $ A = \sup _{x \in \Lambda}\left\{ \log \left\|DG^{-1}\left(x\right)\right\|^{-1}\right\}$. Then $a_{t}\leq A$ for every $t= 0, \cdots , n-1$.
We consider $c_{1}= 2c$ and $c_{2}= 4c$  and then by the Pliss Lemma, there exist $l> \theta n$ with $\theta = \frac{2c}{A - 2c}$ and $l$ integer numbers  $1 \leq n_{1} < \cdots < n_{l} \leq n$  for which the folllowing inequality holds:
$$ a_{k} + \cdots + a_{n_{i}} \geq 2c \left( n_{i}- k \right) \ \ \forall  \ 0\leq k \leq n_{i} -1  \ \ i= 1, \cdots, l. $$
Therefore for each $x \in R ^{n_{i}}$  we have:
\begin{eqnarray*}
\log \left(\displaystyle\prod_{j=k}^{n_{i}-1}\left\|DG\left(G^{j}\left(x\right)^{-1}\right)\right\|^{-1} \right) & = & \displaystyle\sum_{j=k}^{n_{i}-1} \log\left\|DG\left(G^{j}\left(x\right)^{-1}\right)\right\|^{-1} \\
& \geq & \displaystyle\sum_{j=k}^{n_{i}-1} a_{j} \geq 2c \left(n_{i}-k\right).
\end{eqnarray*}
Thus, $n_i$ is a hyperbolic time for $x$ if $x \in R^{n_i}$.
\end{proof}

\begin{Remark}
\label{epsilon} An useful property of hyperbolic times was given in \cite{ABV}. Roughly they proved that, if $n$ is a hyperbolic time for $x$ there exists a neighborhood of  the $n$-th iterate of  $x$ for which one every inverse branch is exponentially contracting. Moreover, the radius $\epsilon > 0$ of the neighborhood depends only on $c$ and on the map.  
\end{Remark}
We will use this property to prove a type of distortion condition for the Birkhoff sums. For the precise statement and the proof see \cite{ABV}.

% Denotaremos por $G_{x,n}^{-j}$ o ramo inverso de $G^{-j}$ que envia $G^{n}(x)$ em $G^{n-j}(x)$ e $B(x, \delta)= \{y \in \Lambda | d(x,y) \leq \delta \}$ a bola fechada de centro $x$ e raio $\delta$.
% 
%\begin{Prop}[ABV]\label{tempos hip}
%Existe $\delta= \delta(G,c) > 0$ (que depende de G e de c) para o qual dado qualquer tempo hiperbólico $n \geq 1$ de um ponto $x \in \Lambda$ e dado $1\leq j\leq n$, o ramo inverso $G_{x,n}^{-j}$ está definido em $B(G^{n}(x), \delta)$ e sastisfaz 
%        $$\left\| DG_{x,n}^{-j}(z)  \right\| \leq e ^{-jc}    \qquad \forall z \in  B(G^{n}(x), \delta).   $$
%\end{Prop}

\begin{Prop}\label{distorcao}
There exists a constant $K_0 > 0$ such that if $x,y$  belong to a hyperbolic cylinder $R^{n}$, the following holds:
$$|S_{n}\phi (x) - S_{n}\phi(y)| \leq K d(x,y)^{\delta} \leq K_{0}  . $$ 
\end{Prop}
\begin{proof}
Assume $\phi$ is H\"older continuous with constants  $(C, \delta)$. Then we have:
\begin{eqnarray*}
   |S_{n}\phi(x) - S_{n}\phi(y)| &  =   & |\displaystyle\sum_{j=1}^{n}\phi(G^{n-j}(x))- \phi(G^{n-j}(y))|  \\
                                 & \leq & \displaystyle\sum_{j=1}^{n}|\phi(G^{n-j}(x))-\phi(G^{n-j}(y))| \\       
                                 & \leq & C\displaystyle\sum_{j=1}^{n} d\big( G^{n-j}(x), G^{n-j}(y)\big)^{\delta}.
\end{eqnarray*}
Also, since $x$ and $y$ belong to the same hyperbolic cylinder, according to the Remark~\ref{epsilon} for each $1 \leq j \leq n$ we have:
             $$d\big( G^{n-j}(x), G^{n-j}(y)\big)^{\delta} \leq e^{-jc}.$$
Then, there is a constant $K_0>0$ which depends only on $\delta, c$ and $C$ which satisfies:
        $$|S_{n}\phi (x) - S_{n}\phi(y)| \leq K_{0}  . $$
\end{proof}

%%%%%%%%%%%%%%%%%%%%%%%%%%%%%%%%%%%%%%%%%%%%%%%%%%%%%%%%%%%%%%%%%%%%%%%%%%%%%%%%%%%%%%%%%%%%%%%%%%%%%%%%%%%%%%%%
%%%%%%%%%%%%%%%%%%%%%%%%%%%%%%%%%%%%%%%%%%%%%%%%%%%%%%%%%%%%%%%%%%%%%%%%%%%%%%%%%%%%%%%%%%%%%%%%%%%%%%%%%%%%%%%%
%%%%%%%%%%%%%%%%%%%%%%%%%%%%   JACOBIAN                        %%%%%%%%%%%%%%%%%%%%%%%%%%%%%%%%%%%%%%%%%%%%%%%%
%%%%%%%%%%%%%%%%%%%%%%%%%%%%                                    %%%%%%%%%%%%%%%%%%%%%%%%%%%%%%%%%%%%%%%%%%%%%%%%
%%%%%%%%%%%%%%%%%%%%%%%%%%%%%%%%%%%%%%%%%%%%%%%%%%%%%%%%%%%%%%%%%%%%%%%%%%%%%%%%%%%%%%%%%%%%%%%%%%%%%%%%%%%%%%%%
%%%%%%%%%%%%%%%%%%%%%%%%%%%%%%%%%%%%%%%%%%%%%%%%%%%%%%%%%%%%%%%%%%%%%%%%%%%%%%%%%%%%%%%%%%%%%%%%%%%%%%%%%%%%%%%%

\subsection{ Jacobians and conformal measures}

\label{s. jacobiano}

\begin{Def} \label{def jacobiano}
 Let $\eta$  be a probability measure on $\Lambda$. We say that a function $J_{\eta}G$ is a \emph{Jacobian} of $\eta$ with respect to $G$ if it satisfies:
$$ \eta \left(G\left(A\right)\right) = \int_{A} J_{\eta}G \, d \eta,   $$ 
for all mensurable set $ A $ such that $ G|_{A} $ is injective.
\end{Def}

In other words, given the measure $\eta$,  we consider $G_{\ast}\eta$ by:
$$G_{\ast}\eta (B) = \eta \big(G(B)\big) \ \mbox{for all  measurable subset} \,  B \,\mbox{of} \ \Lambda .$$
The Radon-Nikodym theorem ensures that the Jacobian is essentially unique. Further, it coincides with the Radon-Nikodym derivative of the measure $G_{\ast}\eta$ with respect to $\eta$, $J_{\eta}G = \frac{dG_{\ast}\eta}{d\eta}$. We will get an useful property regarding the Jacobian with respect to an iterated of $G$.

\begin{Prop} \label{regra da cadeia}
If the Jacobian of $\mu$ with respect to $G$, $J_{\mu}G$, exists, then for each $n \! \in \! \mathbb{N}$ the map $G^n$ also has a Jacobian and it is given by:
          $$J_{\mu}G^{n} = \displaystyle\prod_{j=0}^{n-1} J_{\mu}G \circ G^{j} .$$    
\end{Prop}   

\begin{proof}
For each function $\psi: \Lambda \to \mathbb{R}$ we define the support of  $\psi$ by  $supp(\psi)=\{x \in \Lambda \mid \psi(x)\neq 0 \}$.

For topological spaces, one can establish another definition of Jacobian which is equivalent to the one given in the definition~\ref{def jacobiano}.
We say that  $J_{\mu}G$ is a Jacobian of $G$ with respect to $\mu$, if for any open set $U$ wich satisfying that $G|_{U}$ is  injective, and for all  $\psi \in \mathcal{C}(\Lambda)$ such that $supp(\psi)\subset G(U)$, the following holds:
              $$\displaystyle\int_{G(U)} \psi \ d\mu = \displaystyle\int_{U} (\psi \circ G)J\mu G \ d \mu.$$ 

Using this definition we can prove the proposition. We treat the case $n=2$ since the general case follows by induction.
Given an open set $U$ suppose that $G|_{U}$ is injective and $\psi$ is such that $supp(\psi)\subset G^2(U)$. We have  to show that $J_{\mu}G^{2}= J_{\mu}G \cdot J_{\mu}G \circ G$. In fact, we have:
\begin{eqnarray*}
 \displaystyle\int_{G^{2}(U)} \psi \ d \mu &=&  \displaystyle\int_{G(U)}( \psi \circ G) J_{\mu}G  \ d \mu \\
                                           &=&  \displaystyle\int_{U}\big[ \big[( \psi \circ G )\cdot J_{\mu}G\big]\circ G  \big]\cdot J_{\mu}G \ d \mu \\  
                                           &=&  \displaystyle\int_{U} \big[ ( \psi \circ G ^{2}) \cdot \big( J_{\mu}G\circ G \big)\big] \cdot  J_{\mu}G \ d \mu.   
\end{eqnarray*} 
\end{proof}
In the next lemma we compute the Jacobian of the dual operator eigenmeasure .
\begin{Prop} \label{jacobiano}
If the spectral radio of $\mathcal{L}_{\phi}$,  $\lambda = r \left(\mathcal{L}_{\phi} \right) $, is an eigenvalue associated with an eigenmeasure $\nu$ then we have:
      $$ J_{\nu}G = \lambda e ^{- \phi}.$$ 
(In this case, $\nu$ is said to be a conformal measure.)
\end{Prop}
\begin{proof}
Given $A$ which satisfies $G_{\mid A}$ injective, we have to show that   $\nu (G(A)) = \int_{A} \lambda e^{- \phi} d \nu$.
We have:
\begin{eqnarray*}
\int_{A} \lambda e^{- \phi} \  d \nu & = & \int e^{- \phi} \cdot \chi _{A}  \  d \lambda\nu \\
                                  & = &  \int e^{- \phi} \cdot \chi _{A} \  d \mathcal{L}_{\phi}^{\ast} \nu \\
                                  & = & \int \mathcal{L}_{\phi}\left ( e^{- \phi} \cdot \chi _{A}\right) \  d\nu.    
\end{eqnarray*}
By the injectivity of $G$ in $A$ we get:
\begin{eqnarray*}
\mathcal{L}_{\phi}\left ( e^{- \phi} \cdot \chi _{A}\right)\left( x \right) & = &
 \displaystyle\sum_{y \in G^{-1}\left( x \right)} e^{\phi \left( y \right)} \cdot  e^{-\phi \left( y \right)} \cdot \chi _{A}\left( y \right)  \\ & = & \displaystyle\sum_{y \in G^{-1}\left( x \right)} \chi _{A} \left( y \right)  =  \chi _{G \left( A \right)} \left( x \right)
\end{eqnarray*}
$$ \int_{A} \lambda e^{- \phi} d \nu = \int \chi _{G \left( A \right)}\left( x \right) d\nu = \nu \left( G \left( A \right) \right), $$
wich finishes the proof.
\end{proof}

Thus, combining Proposition~\ref{jacobiano} and Proposition~\ref{regra da cadeia} we obtain:
\begin{Lema} \label{quase Gibbs}
There exists $K_{1} > 0 $ satisfying:
$$ K^{-1}_{1} \leq  \frac{J_{\nu}G^{n}\left( x \right)}{J_{\nu}G^{n}\left( y \right)} \leq K_{1},       $$
for all $x,y \in R ^{n}$ , where $R^n$ is a hyperbolic cylinder.
\end{Lema}
\begin{proof}
Given $n \in \mathbb{N}$, the Proposition~\ref{regra da cadeia} implies that we can calculate the Jacobian of $\nu$ with respect to $G^n$ using the Jacobian of $\nu$ with respect to $G$. Then, according to Proposition~\ref{jacobiano} for each $x \in \Lambda$:
\begin{equation}
   J_{\nu}G^n(x)= \lambda^{n}e^{- S_{n}\phi (x)}.  \label{jac do iterado}
\end{equation}     
Consider $R^n$ a hyperbolic cylinder of length $n$ and $x,y \in R^n$. 
By the Proposition~\ref{distorcao}: 
$$ \left| log \left ( \frac{J_{\nu}G^{n}(x)}{J_{\nu}G^{n}(y)} \right) \right|  = |S_{n}\phi(x) - S_{n}\phi(y)| \leq K_{0}. $$
To finish the proof, consider $K_1 = e^{K_0}$.
\end{proof}

Given $n \in \mathbb{N}$, let $B^{n}$ be the collection of cylinders $R^n(a_0, \cdots, a_{n-1})$ generated by words $(a_0, \cdots, a_{n-1}) \not\in I(\gamma,n)$. 

Let $P^{n}$ be the collection of cylinders generated by words of length $n$ belonging to $I(\gamma,n)$.  

\begin{Prop}\label{dec exp}
The measure of the cylinders in $P^n$, with respect to $\nu$, decreases exponentially when $n$ goes to infinity. Therefore, $\nu$ almost every point $x \in \Lambda$ belongs to elements of $B^{n}$ for infinitely many values of $n$.
\end{Prop}
\begin{proof}
Let $n\! \in \! \mathbb{N}$ and take $R^{n}\in P^{n}$ a cylinder.
Using \eqref{jac do iterado}, we have:
\begin{eqnarray*}
    \nu \big(G^{n}(R^{n})\big)&  =   & \displaystyle\int_{R^{n}} J_{\nu}G^{n} \ d\nu \\
                              &  =   & \displaystyle\int_{R^{n}} \lambda^{n}e^{-S_{n}\phi} \ d\nu \\
                              & \geq & \lambda^{n}e^{-{n}\sup\phi} \nu(R^{n}).
\end{eqnarray*}
Whereby $\nu(G^{n}(R^{n})) \leq 1$ it implies $\nu(R^{n}) \leq  \lambda^{-n}e^{{n}\sup\phi}$.
If we take the sum over all cylinders generated by words $(a_0, \cdots, a_{n-1}) \in I(\gamma,n) $ and use the Proposition~\ref{cilindros ruins}, we obtain:
        $$ \displaystyle\sum _{I(\gamma,n)} \nu(R^{n}) \leq \lambda^{-n}e^{\sup\phi + \beta}.  $$
On the other hand,  Lemma~\ref{cota lambda} gives us $\lambda \geq \exp(\inf \phi + \log \omega)$. Since $\beta = \frac{\log(\omega)}{2}$, we have the result.
\end{proof}

We use the notation $R_h^{n}$ to refer to hyperbolic cylinders of length $n$.
In the Proposition~\ref{temp hiperbolicos} we obtained plenty of  hyperbolic times associated to words which do not belong to the set $I(\gamma,n)$.

Let $\mathcal{H}$ be the closure of the set of points $x\in \Lambda$ which belong to infinitely many hyperbolic cylinders $R_{h}^{n}$. Next, we prove that $\mathcal{H}$ is a $\nu$ full measure set with respect to $\nu$. 

\begin{Prop}\label{H total}
Let $n_1(x)$ be the first hyperbolic time of $x$. We have that  $n_{1}$ is integrable and $\nu(\mathcal{H})= 1$.
\end{Prop}
\begin{proof}
By the Proposition~\ref{temp hiperbolicos}, for each $x \in B^{n}$ we have $n_{1}(x) \leq n$. Therefore, if $x$ is such that $n_{1}(x) > n$ we have $x \in P^{n}$. Then:
     $$ \int n_{1} \, d\nu = \displaystyle\sum_{n=0}^{\infty} \nu\big( \{ x: n_{1}(x) > n \}   \big) \leq 1 +  \displaystyle\sum_{n=1}^{\infty} \nu(P^{n}) . $$
According to the Proposition~\ref{dec exp}, the measure of the sets in $P^{n}$ decreases exponentially  with $n$, which implies that the sum in the last inequality is finite. Then, the first hyperbolic time is integrable with respect to $\nu$.
Since $B^{n}$ consists of cylinders generated by words in the complement of  $I(\gamma, n)$, the Proposition~\ref{temp hiperbolicos} guarantees that every point  belongs to at least $l$ hyperbolic cylinders, where $l > \theta n $. Besides that, by the Proposition~\ref{dec exp}, for $n$ large enough, $\nu$ almost every point belongs to a set in $B^n$. Thus, $\nu$ every point $x$ in $\Lambda$ belongs to a infinitely many hyperbolic cylinders, then, $\nu(\mathcal{H}) = 1$.
\end{proof}

We consider for each $x \in \mathcal{H}$, $n_k(x)$ the $k$-th hyperbolic time of $x$. We also consider $\theta$ as in the Proposition~\ref{temp hiperbolicos}. According to the proof of Proposition~\ref{H total} we have the following corollary:

\begin{Cor}\label{teta pliss}
For $\nu$ every point $x \in \Lambda$ and $n$ large enough the following holds:
    $$ \# \{k\geq 1 : n_{k}(x) \leq n  \} \geq \theta n .$$
\end{Cor}

In the next proposition we estimate the measure of hyperbolic cylinders with respect to $\nu$. This property will  be useful later.
%
%\subsection{Non-lacunary Gibbs Measures}
%
%In this section we will prove that the transfer measure obtained in the Theorem~\ref{Teorema B} satisfies a Gibbs type condition.
%
%We will prove that $\nu$ is a Non-lacunary Gibbs measure. This concept was introduced by Oliveira and Viana in \cite{OV05}.

\begin{Prop}\label{Gibbs}
There exists a constant $K_{2} > 0$ such that for every $x \in R^{n}_h$, a hyperbolic cylinder of length $n$, and every $n \in \mathbb{N}$, we have:
$$   K ^{-1}_{2} \leq \frac{\nu (R^{n})}{\exp(S_{n}\phi(x) - Pn)}\leq K_{2}, $$
where $ P= log \lambda.$
\end{Prop}

\begin{proof}
Let $n \! \in \! \mathbb{N}$, $R^{n} = R^{n}(a_{0},a_{1},\cdots, a_{n-1})$ and $x \! \in \! R^{n}$.
By the Proposition~\ref{med referencia} we have 
$$ \left( \mathcal{L}_{\phi}^{\ast}\right)^{n} \nu = \lambda^{n}\nu \quad \forall n \in \mathbb{N}.  $$
And then, we obtain
\begin{eqnarray*}
   e^{Pn} \nu \left(R ^{n} \right) & = & e^{Pn} \int \chi_{R ^{n}}  d\nu \\
                                   & = & \frac{e^{Pn}}{\lambda^{n}}\int \mathcal{L}_{\phi}^{n} \left( \chi_{R ^{n}}\right) d \nu   \\
                                   & = & \int \mathcal{L}_{\phi}^{n} \left( \chi_{R ^{n}}   \right) d \nu,
\end{eqnarray*}
where $\chi_{R^{n}}$ denotes the characteristic function on $R^n$ and $e ^{P}= \lambda$.
According to the transfer operator definition we can write: 
$$ \mathcal{L}_{\phi}^{n}  \chi_{R ^{n}}(x)= \displaystyle\sum _{y \in G^{-n}(x)} e^{S_{n}\phi (y)}  \chi_{R ^{n}}(y) =  \displaystyle\sum _{y \in G^{-n}(x)\cap R^{n}} e^{S_{n}\phi (y)}   .$$
Therefore, given $n \in \mathbb{N}$ for each $x \in R^{n}$ there exists a unique element $y \in G ^{-n}(x)\cap R^{n}(x)$. We define $A(n)= \{y \in G^{-n}(x)\cap R ^{n}\}$ and then:
$$ \left| \frac{\nu (R^{n})}{\exp\left(S_{n}\phi(x) - Pn \right)} \right|  = \left| \frac{ \int_{A(n)}e^{S_{n}\phi}d \nu}{e^{S_{n}\phi(x)}} \right|.$$
By the Proposition~\ref{distorcao} there exists a constant $K_0>0$ satisfying:
$$\mid S_{n}\phi(x) - S_{n}\phi(y) \mid < K_0 ,  $$
for all x, y that belong to the same hyperbolic cylinder, which is enough to finish the proof.
\end{proof}

%%%%%%%%%%%%%%%%%%%%%%%%%%%%%%%%%%%%%%%%%%%%%%%%%%%%%%%%%%%%%%%%%%%%%%%%%%%%%%%%%%%%%%%%%%%%%%%%%%%%%%%%%%%%%%%%
%%%%%%%%%%%%%%%%%%%%%%%%%%%%%%%%%%%%%%%%%%%%%%%%%%%%%%%%%%%%%%%%%%%%%%%%%%%%%%%%%%%%%%%%%%%%%%%%%%%%%%%%%%%%%%%%
%%%%%%%%%%%%%%%%%%%%%%%%%% OPERADOR DE TRANSFERENCIA               %%%%%%%%%%%%%%%%%%%%%%%%%%%%%%%%%%%%%%%%%%%
%%%%%%%%%%%%%%%%%%%%%%%%%%%%%%%%%%%%%%%%%%%%%%%%%%%%%%%%%%%%%%%%%%%%%%%%%%%%%%%%%%%%%%%%%%%%%%%%%%%%%%%%%%%%%%%%
%%%%%%%%%%%%%%%%%%%%%%%%%%%%%%%%%%%%%%%%%%%%%%%%%%%%%%%%%%%%%%%%%%%%%%%%%%%%%%%%%%%%%%%%%%%%%%%%%%%%%%%%%%%%%%%%
\subsection{Transfer Operator Eigenfunction}
\label{autofuncao}

In this section we will show the existence of an eigenfunction to the transfer operator of Ruelle-Perron-Fr\"obenius. We will, therefore, show the second part of the Theorem~\ref{Teorema B} and finish its proof. Besides that, we will show that any eigenfunction is positive and H\"older continuous.

\begin{Teo}
There exists an eigenfunction $h : \Lambda  \longrightarrow  \mathbb{R}$ associated to the operator $\mathcal{L}_\phi$. Moreover, it is positive and associated to the eigenvalue $\lambda = r(\mathcal{L}_ \phi)$: 
$$\mathcal{L}_{\phi} h = \lambda h.$$ 
\end{Teo}
We define, for each $n \in \mathbb{N}$, the function $T_{n}: \Lambda \longrightarrow  \mathbb{R}$ by: 
\begin{eqnarray*}     
           T_{n}(x)=  \displaystyle\sum_{{ y \in G^{-n}(x)\cap R^{n}}\atop
           {R^{n} \ hyperbolic} }^{}
             e^{S_{n}\phi(y)}  \qquad \forall x \in \Lambda . 
\end{eqnarray*} 
We also define, for each $n \in \mathbb{N}$:
   $$ h_{n} = \frac{1}{n}\sum_ {i=1}^{n-1}{\lambda} ^{-i}T_{i} . $$
Consider the following sequence of real numbers: 
$$Z_{n}= \sum_{R^{n} \, hyperbolic} e^{sup\left\{S_{n}\phi(x) \mid x \in R^{n} \right\}}.$$
Note that for every $x \in \Lambda$ we have: 
        $$ T_{n}(x) \leq  Z_{n}           .$$
\begin{Lema}\label{limitacao}
There exists a constant $K_{3} > 0$ satisfying
   $$\lambda^{-n}T_{n} \leq \lambda^{-n} Z_{n} \leq  K_{3}\quad \mbox{and} \quad  K_{3}^{-1} \leq \frac{1}{n}\displaystyle\sum_{j=0}^{n-1} \lambda^{-j}Z_{j}.     $$   
\end{Lema}
\begin{proof}
The first inequality is trivial because $\lambda > 0$. So, it is enough to get an upper bound for $\lambda^{-n}Z_{n}$.
Let $n \in \mathbb{N}$, $R^{n}_{h}$ a hyperbolic cylinder and $y \in R_{h}^{n}$. The Proposition~\ref{Gibbs} ensures that 
  $$K_{2}^{-1} \nu(R^{n}) \leq {\lambda}^{-n} e ^{S_{n}\phi (y)}  \leq K_{2} \nu(R^{n}) . $$
By taking the supremum over every $y \in R^{n}$ and taking the sum over all hyperbolic cylinder $R ^{n}$ of length $n$, we have:
\begin{eqnarray}\label{eq}
K_{2}^{-1}\sum_{R^{n} hip.}\nu(R^{n}) \leq \sum_{R^{n} hip.} \lambda^{-n} e ^{sup S_{n}\phi}  \leq  K_{2}\sum_{R^{n} hip.}\nu(R^{n})    .   
\end{eqnarray}
Thus
$$  {\lambda}^{-n}Z_{n} \leq  K_{2} \cdot \nu (\displaystyle\bigcup_{R^{n} hip} R^{n}) \leq  K_{2}.   $$
To show the second part of the Lemma we take the first inequality in \eqref{eq}, and then:
\begin{eqnarray*}
   K_{2}^{-1}\frac{1}{n}\displaystyle\sum_{j=0}^{n-1} \nu(\displaystyle\bigcup_{R^{j} hip} R^{j}) \leq \frac{1}{n}\displaystyle\sum_{j=0}^{n-1} \lambda^{-j}Z_{j} 
\end{eqnarray*}
For every $j\in \mathbb{N}$ let $\mathcal{R}_{j}:= \displaystyle\bigcup_{R^{j}_ hip} R^{j}$ be the union of hyperbolic cylinders of lenght $j$. According to Corollary~\ref{teta pliss}, for almost every point $x \in \Lambda$ and $n$ large enough, we have: 
\begin{eqnarray*}
\frac{1}{n}\displaystyle\sum_{j=0}^{n-1}\nu ({\mathcal{R}}_{j}) & = & \int\int \chi_{\mathcal{R}_{j}}(x) d\nu(x) \, dm_n(j)\\ 
                                                     & = & \int\int \chi_{\mathcal{R}_{j}}(x) \, dm_n(j) d\nu(x) \geq \theta > 0. 
\end{eqnarray*}
Note that we use Fubini's Theorem in the second equality.
To finish the prove, take $K_3 = \frac{K_2}{\theta}$.

\end{proof}

%%%%%%%%%%%%%%%%%%%%%%%%%%%%%%%%%%%%%%%%%%%%%%%%%%%%%%%%%%%%%%%%%%%%%%%%%%%%%%%%%%%%%%%%%%%%%%%%%%%%%%%%%%%%%%%%
%%%%%%%%%%%%%%%%%%%%%%%%%%%%%%%%%%%%%%%%%%%%%%%%%%%%%%%%%%%%%%%%%%%%%%%%%%%%%%%%%%%%%%%%%%%%%%%%%%%%%%%%%%%%%%%%
%%%%%%%%%%%%%%%%%%%%%%%%%%%%   Equicontinuidade     %%%%%%%%%%%%%%%%%%%%%%%%%%%%%%%%%%%%%%%%%%%%%%%%%%%%%%%                                    
%%%%%%%%%%%%%%%%%%%%%%%%%%%%%%%%%%%%%%%%%%%%%%%%%%%%%%%%%%%%%%%%%%%%%%%%%%%%%%%%%%%%%%%%%%%%%%%%%%%%%%%%%%%%%%%%
%%%%%%%%%%%%%%%%%%%%%%%%%%%%%%%%%%%%%%%%%%%%%%%%%%%%%%%%%%%%%%%%%%%%%%%%%%%%%%%%%%%%%%%%%%%%%%%%%%%%%%%%%%%%%%%%

\subsection{Equicontinuity}

Let $\delta$ be the H\"older constant of the potential $\phi$ in Theorem~\ref{Teorema Principal}. We will prove that family $h_n$ is equicontinuous.
\begin{Lema}\label{Tn Holder}
There exists a constant $K_{4}>0$ such that for $n \in \mathbb{N}$, the function ${\lambda}^{-n}T_{n}$ is H\"older continuous with constants $K_4$ and $\delta$.
\end{Lema}
\begin{proof} 
Since the rectangles $R_i$, $i=1,2$ or $3$, are  two-by-two disjoints it is enough to consider points in the same rectangle. Let $n \in \mathbb{N} $ and $x_{1}, x_{2} \in R_i$. Then $\left| \lambda ^{-n} T_{n}(x_{1}) - \lambda^{-n}T_{n}(x_{2})\right| $ is equal to
$$ \left| \lambda^{-n} \left(  \sum_{y \in G^{-n}(x_{1})\cap R_{h}^{n}} e ^{S_{n}\phi(y)} -  \sum_{z \in G^{-n}(x_{2})\cap R_{h}^{n}} e ^{S_{n}\phi(z)} \right) \right| .$$
We group terms whose pre-images of $x_{1}$ and $x_{2}$ belong to the same hyperbolic cylinder.  Thus, if $y \in G^{-n}(x_{1})$ and $z \in G^{-n}(x_{2})$ belong to the same hyperbolic cylinder of length $n$, by the Proposition~\ref{distorcao}:
$$ \mid S_{n}\phi (y) - S_{n}\phi(z)  \mid < K_{0} d (x_{1},x_{2})^{\delta} .$$
On the other hand, we can write:
\begin{eqnarray*}
\mid e^{S_{n}\phi(y)} -  e^{S_{n}\phi(z)} \mid  &\leq & e^{max\{S_{n}\phi(y),S_{n}\phi(z) \}} \mid S_{n}\phi(y) - S_{n}\phi(z)\mid \\
                                                &\leq & e^{max\{S_{n}\phi(y),S_{n}\phi(z) \}} K_{0}d(x_{1},x_{2})^{\delta}.
\end{eqnarray*}   
By the definiton of $Z_n$:
\begin{eqnarray*}
\sum_{y \in G^{-n}(x_{1})\cap R_{h}^{n}} e ^{S_{n}\phi(y)} -  \sum_{z \in G^{-n}(x_{2})\cap R_{h}^{n}} e ^{S_{n}\phi(z)} &\leq &\\
\leq d(x_1,x_2)^{\delta}K_0 \displaystyle\sum_{z,y \in R_{h}^{n}}e^{\max \{ S_{n}\phi(y),S_{n}\phi(z) \}}
              &\leq &  d(x_1,x_2)^{\delta}K_{0} Z_{n}.\end{eqnarray*} 
Thus
 $$\mid {\lambda}^{-n}T_{n}(x_{1}) - {\lambda}^{-n}T_{n}(x_{2})\mid  \leq \lambda^{-n}Z_{n}K_{0}d(x_1,x_2)^{\delta} \leq K_{4}d(x_1, x_2)^{\delta}, $$ 
where we consider $K_3$ as in the Lemma~\ref{limitacao} and $K_4= K_{0}K_{3}$.
\end{proof}

\begin{Prop}\label{infernal}
The sequence $\mid\lambda^{-k} (\mathcal{L}_{\phi} T_{k} - T_{k+1}) \mid$ converges uniformly to zero when $k$ goes to infinity.
\end{Prop}

\begin{proof}
Given $k\geq 0 $ and $x \in \Lambda$, by definition we have:
\begin{eqnarray*}
 \mathcal{L}_{\phi}T_{k}(x)& = & \displaystyle\sum_{y \in G^{-1}(x)} e ^{\phi(y)}T_{k}(y)= \\
                           & = &  \displaystyle\sum_{y \in G^{-1}(x)} e ^{\phi(y)}  \displaystyle\sum_{z \in G^{-k}(y)\cap R_{h}^{k}} e ^{S_{k}\phi(z)}.
\end{eqnarray*}
We consider the set  $A_{k+1}$ of cylinders 
$ R^{k+1}[a_0,a_1,\cdots,a_k]$  that are not  hyperbolic, such that $ R^{k}[a_1,..., a_k]$ are hyperbolic.
In other words, $A_{k+1}$ is the set of every cylinder of length $k+1$, such that considering the $k$ last symbols which generated the cylinder we have that the cylinder of length $k$ is hyperbolic but the cylinder of length $k+1$ is not hyperbolic. 

Note that there is no $0\leq i \leq k$  that is a hyperbolic time for every element of $R^{k+1}$. Otherwise, by concatenation, $R^{k+1}$ would be hyperbolic, which does not happen.
By definition of $T_{k}$:
\begin{eqnarray*}
\displaystyle\mid \lambda^{-k} (\mathcal{L}_{\phi} T_{k} - T_{k+1}) \displaystyle\mid & \leq & \displaystyle\mid \lambda^{-k}
 \displaystyle\sum _{{z \in G^{-1}(y)\cap A}\atop
           {A \in A_{k+1}} }^{}
 e^{S_{k+1}\phi(z)}\displaystyle\mid \\
                                                            & \leq & \displaystyle\mid \lambda^{-k} \cdot \# A_{k+1} \cdot e ^{(k+1)sup\phi} \displaystyle\mid.
\end{eqnarray*}

However, $ A_{k+1} \subset \{R^{k+1}\left(a_0, \cdots, a_k \right) \backslash \left(a_{0}, \cdots,a_{k}  \right) \in I\left( \gamma, k+1 \right)  \}$, where $\gamma$ is given by the Proposition~\ref{cilindros ruins}.
In fact, since no cylinder $R^{i}(a_0,..., a_i)$ is hyperbolic, the  proportion of indices $i \in \{0, \cdots k+1  \}$ such that $(a_0, \cdots, a_{k+1})$ is hyperbolic is zero, thus smaller than $\theta$. See Proposition~\ref{teta pliss}. 
Then, by the Proposition~\ref{cilindros ruins} we have:
   $$\# A_{k+1} \leq e^{ \beta(k+1)} \ \mbox{for} \ k \ \mbox{large enough}.$$
Remember that we take $\beta =\frac{\log(\omega)}{2}$.
From the Lemma~\ref{cota lambda}, we have that $\lambda \geq e^{ \inf \phi + \log \omega}$, thus:
\begin{eqnarray*}
\displaystyle\mid \lambda^{-k} (\mathcal{L}_{\phi} T_{k} - T_{k+1})\displaystyle\mid & \leq & e^{- k(\inf \phi + \log \omega)} \cdot  e^{\beta (k+1)} \cdot  e^{(k+1)\sup \phi}\\
                          &  =   & e^{-k(-\beta - (\sup \phi - \inf \phi) + \log \omega)} \cdot e^{ \beta + \sup \phi}.
\end{eqnarray*}
Since $\sup\phi - \inf \phi < \beta$ we have that the last term in the inequality converges exponentially to zero when $k$ goes to $\infty$.
\end{proof}
Recall that:
   $$ h_{n}(x) = \frac{1}{n}\sum_ {i=0}^{n-1}{\lambda} ^{-i}T_{i}(x) \quad \forall x \in \Lambda.  $$ 
According to the Lemma~\ref{Tn Holder} every function $\lambda^{-n}T_n$ is $(K_4, \delta)-$ H\"older continuous, but since $h_{n}$ is the average of the functions $\lambda^{-n}T_n$ we have that $h_n$ is also $(K_4, \delta)-$ H\"older continuous and, then, the family $(h_n)_{n\in \mathbb{N}}$ is equicontinuous.

\begin{Lema}\label{h acumulacao}   
The sequence of functions $(h_{n})_{n\in \mathbb{N}}$ has at least one accumulation point. Further, any accumulation point $h$ satisfies:
            $$\mathcal{L}_{\phi}h = \lambda h  .    $$
\end{Lema}
\begin{proof} 
The Lemma~\ref{limitacao} ensures that the $h_{n}$ is uniformly bounded, since it is also equicontinuous, by the Ascoli-Arzelá Theorem we have existence of an accumulation point. Let $h$ denote one accumulation point:
    $$h=\displaystyle\lim_{k \to \infty} h_{n_k} .            $$
We still have to show that any accumulation point $h$ is an eigenfunction. In fact:  
\begin{eqnarray*}
\mathcal{L}_{\phi}h & = & \mathcal{L}_{\phi}(\displaystyle\lim_{k \to \infty} h_{n_k}) = \displaystyle\lim_{k \to \infty}\mathcal{L}_{\phi}( h_{n_k}) \\
                    & = & \displaystyle\lim_{k \to \infty} \frac{1}{n_k}\left[ \mathcal{L}_{\phi}T_0+ \lambda^{-1}\mathcal{L}_{\phi}T_1+\dots + \lambda^{-n_{k}+1}\mathcal{L}_{\phi}T_{n_{k}-1}\right] \\
                    & = &\displaystyle\lim_{k \to \infty} \left[ \frac{1}{n_k}\left[ \lambda (T_0+\dots+\lambda^{n_{k}-1}T_{n_{k}-1})-\lambda T_0 +\lambda^{-n_{k}+1}T_{n_k} \right] \right.\\
                    &   & \left. -\frac{1}{n_k}\sum_{i=1}^{n_k-1}{\lambda}^{-i}(\mathcal{L}_{\phi}T_{i}-T_{i+1}) \right].
\end{eqnarray*}
By the Proposition~\ref{infernal} the last term converges to zero. Besides that, the Lemma~\ref{limitacao} guarantees that $\lambda^{-n}T_{n}$ is uniformly bounded, then the correspondent term in the limit also converges to zero. Then, we have:
     $$\mathcal{L}_{\phi}h = \displaystyle\lim_{k \to \infty}  \frac{1}{n_k} [ \lambda (T_0+\dots+\lambda^{n_k-1}T_{n_k-1}) ] = \lambda h.     $$ 
\end{proof}

\begin{Lema}
There exist a constant $K_5>0$ such that
   $$Z_n \leq K_{5}T_n(x), \qquad \mbox{for every}\ n\in \mathbb{N} \ \mbox{and for every}\ x\in \Lambda.       $$
\end{Lema}
\begin{proof}
Let $x\in \Lambda$ and $n\geq 1$, and fix $m>3$. Note that every rectangle $R_{i}$, $i=1,2$ or $3$, cover $\Lambda$ after three iterations. Let $y,z \in G^{-n}(x) \cap R_h^n$, by  Proposition~\ref{distorcao}, since  $y,z$ belong to the same hyperbolic cylinder, we have:
      $$e^{S_{n}\phi(y)} \leq e^{ K_0}e^{ S_{n}\phi(z)}.$$
On the other hand, one can write $$S_{n}\phi(y) = S_{n+m}\phi (y) - S_{m}\phi(G^{n}(y)) \leq S_{n+m}\phi (y) - m \cdot\inf\phi .$$
If we fix $y$ and take the supremum over every $w$ in the hyperbolic cylinder $R^{n}$ we get:
     $$e^{\sup S_{n}\phi|_ {R^{n}}} \leq e^{K_0} e^{-m \cdot \inf \phi}e^{S_{n+m}\phi(y)}.  $$
Note that each cylinder of length $n+m$ is, naturally, included in a cylinder of length $n$. Adding over every hyperbolic cylinder of length $n$, we obtain:
        $$Z_{n}\leq e^{ K_0} e^{-m \cdot \inf \phi}\displaystyle\sum e^{S_{n+m}\phi(y)} \leq e^{ K_0} e^{-m \cdot \inf \phi} T_{n+m}(x) . $$
Comparing $T_{n+m}(x)$ and $T_n(x)$, we have       
\begin{eqnarray*}          
T_{n+m}(x) & =   & \displaystyle\sum_{y \in G^{-(n+m)}(x)\cap R_{h}^{n + m}}e^{S_ {n+m}\phi(y)}\\
           &\leq & \displaystyle\sum_{y \in G^{-(n+m)}(x)\cap R_{h}^{n + m}}  e^{S_{n}\phi(G^{m}(y)) + S_{m}\phi(y)}\\
           &\leq & \log(\omega)e^{m \sup \phi}  \displaystyle\sum_{z \in G^{-(n)}(x)\cap R_{h}^{n}} e^{S_{n}\phi(z)}\\
           &\leq &  \log(\omega)e^{m \sup \phi} T_{n}(x) .       
\end{eqnarray*} 
Let $K_5 = \log(\omega)e^{K_0 +m (\sup \phi - \inf \phi)}$.
Thus 
       $$ Z_{n} \leq K_5 T_{n}(x),  \quad \mbox{for all} \ x \in \Lambda.        $$
\end{proof}

Using Lemma~\ref{limitacao}, we have
$$K_3^{-1} \leq \frac{1}{n} \sum_{j=0}^{n-1} \lambda^{-j}Z_j  \leq  K_5 \frac{1}{n} \sum_{j=0}^{n-1} \lambda^{-j}T_j(x) =K_5 h_n(x). $$
We just proved the next lemma:   
\begin{Lema} \label{h positivo}
The eigenfunction of the Transfer Operator obtained as a accumulation point of the sequence $(h_{n})_{n \in \mathbb{N}}$ satisfies $h > 0$.       
\end{Lema}
We obtained, in this section, an eigenfunction for the Transfer Operator, and an eigenmeasure for its dual, both associated with the same eigenvalue, and it finishes the proof of Theorem~\ref{Teorema B}. Moreover we prove that the eigenfunction is positive. 
%%%%%%%%%%%%%%%%%%%%%%%%%%%%%%%%%%%%%%%%%%%%%%%%%%%%%%%%%%%%%%%%%%%%%%%%%%%%%%%%%%%%%%%%
%%%%%%%%%%%%%%%%%%%%%%%%%%%%%%%%%%%%%%%%%%%%%%%%%%%%%%%%%%%%%%%%%%%%%%%%%%%%%%%%%%%%%%%%%%%%%%%%%%%%%%%%%%%%%%%%
%%%%%%%%%%%%%%%%%%%%%%%%%%%%%%%%%%%%%%%%%%%%%%%%%%%%%%%%%%%%%%%%%%%%%%%%%%%%%%%%%%%%%%%%%%%%%%%%%%%%%%%%%%%%%%%%
%%%%%%%%%%%%%%%%%%%%%%%%%%%%   EXISTENCIA                       %%%%%%%%%%%%%%%%%%%%%%%%%%%%%%%%%%%%%%%%%%%%%%%%
%%%%%%%%%%%%%%%%%%%%%%%%%%%%                                    %%%%%%%%%%%%%%%%%%%%%%%%%%%%%%%%%%%%%%%%%%%%%%%%
%%%%%%%%%%%%%%%%%%%%%%%%%%%%%%%%%%%%%%%%%%%%%%%%%%%%%%%%%%%%%%%%%%%%%%%%%%%%%%%%%%%%%%%%%%%%%%%%%%%%%%%%%%%%%%%%
%%%%%%%%%%%%%%%%%%%%%%%%%%%%%%%%%%%%%%%%%%%%%%%%%%%%%%%%%%%%%%%%%%%%%%%%%%%%%%%%%%%%%%%%%%%%%%%%%%%%%%%%%%%%%%%%

\section{Existence of equilibrium states}\label{Existence2}
In this section we consider the probability measure $\mu:= h \nu$, where $\nu $ is the reference measure from the previous section. More precisely, $\mu$ is defined in every Borelian $A$  by: 
            $$\mu(A) = \displaystyle\int _{A} h \ d\nu .              $$ 
Since $h>0$, the measures $\mu$ and $\nu$ are equivalent.            
We will prove $\mu$ is an equilibrium state for $G$ with respect to the potential $\phi$.  Some of the statements and proofs are analogous to the ones in  \cite{OV05} and \cite{LOR}, and we adapt them to our context, omitting some of the proofs.

%\begin{Lema}
%The probability measure $\mu$ is equivalent to $\nu$.
%\end{Lema}
%\begin{proof}
%In fact, $\mu << \nu$, for given a Borelian $A$ satisfying $\nu(A)= 0$ we have:
%$$\mu (A) = \displaystyle\int_{A} h \ d\mu \leq \nu(A) \cdot \sup h = 0.$$
%On the other hand, if $B$ is a Borelian with $\mu(B)=0$, so
%             $$ \displaystyle\int_{B}h \ d\nu =0    .$$
%But by the Lemma~\ref{h positivo} , $h > 0$, then $\nu(B)=0$.           
%\end{proof}

\begin{Lema}\label{med invariante}
The probability measure $\mu$ is invariant by $G$.
\end{Lema}
\begin{proof}
We will show that for every continuous function $\psi \in \mathcal{C}^{0}(\Lambda)$ we have: 
             $$\int \psi \circ G \ d\mu = \int \psi \ d \mu .         $$
In fact:
\begin{eqnarray}\label{mu inv}
\int \psi \circ G \ d\mu & = & \int( \psi \circ G) \cdot h \ d\nu  \label{mu inv 1} \\
                         & = & \frac{1}{\lambda}\int( \psi \circ G) \cdot h \ d\mathcal{L}^{\ast}\nu   \label{mu inv 2} \\
                         & = & \frac{1}{\lambda}\int \mathcal{L}_{\phi}\left[( \psi \circ G) \cdot h \right] \ d\nu  \label{mu inv 3} \\
                         & = & \frac{1}{\lambda}\int \psi \cdot \mathcal{L}_{\phi}(h) \ d\nu  \label{mu inv 4} \\
                         & = &  \int \psi \cdot h \ d \nu = \int \psi \ d\mu \label{mu inv 5}  
\end{eqnarray}            
Note that \eqref{mu inv 2} and \eqref{mu inv 3} follows from the fact that $\nu$ is an eigenmeasure and $h$ is an eigenfunction. We get \eqref{mu inv 4} from the definition of the Transfer operator. To finish, \eqref{mu inv 5} follow again from the fact that $h$ is an eigenfunction and by the definition of $\mu$.           
\end{proof}

We say that a sequence $(a_j)$ of integers is {\it non-lanunary } if it is increasing and satisfies $\frac{a_{j+1}}{a_{j}} \to 1$.

\begin{Def}
An invariant probability measure $\eta$ is said to be a non-lacunary Gibbs measure if there exist a constant K> 0 and a real number $P$ satisfying
$$K^{-1} \leq \frac{\eta(R^{a_j}(x))}{\exp(S_{a_j}\phi(x) - P{a_j})}\leq K, $$
for $\eta$ almost every $x$  and for a non-lacunary sequence $a_j= a_j(x)$, where $R^{a_j}(x)$ is the cylinder of length $a_j$ contaning $x$.
\end{Def}

\begin{Prop}\label{seq nao-lacunar}
The sequence of hyperbolic times $n_k(x)$ is a non-lacunary sequence for $\mu$ almost every $x \in \Lambda$. 
\end{Prop}
\begin{proof}
According to the Proposition~\ref{H total}  we have that $n_{1}$ is  $\nu$-integrable. Since $h$ is bounded, $n_{1}$ is $\mu$-integrable. Besides, $\mu$ is invariant and $\mu(\mathcal{H})= 1$, by the Proposition~\ref{H total}. Following Proposition 3.8 in \cite{OV08}, we have that $n_{k}(x)$ is non-lacunary for $\mu$ almost every point. 
\end{proof}

 \begin{Prop}\label{mu Gibbs}
The probability $\mu$ is an invariant non-lacunary Gibbs measure.
\end{Prop}
\begin{proof}
Combining the Proposition~\ref{Gibbs} and the fact that $h$ is bounded, both from below and above, we have that there exists $K>0$ satisfying:
$$K^{-1} \leq \frac{\mu(R^n)}{\exp(S_{n}\phi(x) - Pn)}\leq K, $$
for every hyperbolic cylinder $R^{n}$  and every $x \in R^n$.
By the Proposition~\ref{seq nao-lacunar}, the sequence of hyperbolic times is a non-lacunary sequence, and then, the proposition is proved.
\end{proof}

\begin{Lema}\label{jacobiano de mu}
The Jacobian of $\mu$ with respect to $G$ is given by:
    $$J_{\mu}G = \frac{ h \circ G}{h} J_{\nu}G  . $$
In particular,  $J_{\mu}G = \lambda \frac{h \circ G}{h} e^{- \phi}$.
\end{Lema}
\begin{proof}

Given a  Borelian $A$ where $G$ is injective, we have to show that: 
    $$\mu\big(G(A)\big)= \int_{A} \frac{h \circ G}{h} J_{\nu}G . $$
According to Proposition~\ref{jacobiano} we have $J_{\nu}G = \lambda e^{- \phi}$.    
Consider $(\psi_{n})_{\{ n \geq 1 \}}$ a sequence of continuous functions in $\Lambda$ such that $\psi_{n} \rightarrow \chi_{A} \ \ \nu$-a.e.p. and $\parallel \psi_{n} \parallel \leq 2 \ \forall n \geq 1$ where $\parallel\psi_{n} \parallel = \displaystyle\sup_{|x|=1} |\psi(x)| $.
Thus:
$$\mathcal{L}_{\phi}(e^{-\phi}\psi_{n})(x) = \displaystyle\sum_{y \in G^{-1}(x)}e^{\phi(y)}e^{- \phi(y)}\psi_{n}(y)= \displaystyle\sum_{y \in G^{-1}(x)} \psi_{n}(y).$$

Note that the last term converges to $\chi _{G(A)}(x)$ for $\mu$ almost every point $x$, since $A$ is such that $G$ is injective in  $A$.

Then the Dominated Convergence Theorem allows us to conclude that 
   $$\displaystyle\int_{\Lambda} \chi_{G(A)} d\mu = \displaystyle\lim_{n \rightarrow +\infty} \int \mathcal{L}_{\phi}(e^{- \phi}\psi_{n}) \ d\mu   . $$
Therefore, by the Theorem~\ref{Teorema B}, we obtain: 
\begin{eqnarray*}
\mu\big(G(A)\big) & = & \displaystyle\lim_{n \rightarrow +\infty }\displaystyle\int_{\Lambda}\mathcal{L}_{\phi}(e^{-\phi}\psi_{n}) \ d\mu \\
          & = & \displaystyle\lim_{n \rightarrow +\infty }\displaystyle\int_{\Lambda}\mathcal{L}_{\phi}(e^{-\phi}\psi_{n})h \ d\nu \\
          & = & \displaystyle\lim_{n \rightarrow +\infty }\displaystyle\int_{\Lambda}\mathcal{L}_{\phi}((h \circ G)e^{-\phi}\psi_{n}) \ d\nu \\
          & = & \displaystyle\lim_{n \rightarrow +\infty }\displaystyle\int_{\Lambda}(h \circ G)(e^{-\phi}\psi_{n}) \ d\mathcal{L}_{\phi}^{\ast}\nu \\
          & = & \displaystyle\lim_{n \rightarrow +\infty }\displaystyle\int_{\Lambda}\lambda (h \circ G)(e^{-\phi}\psi_{n}) \ d\nu \\
           & = & \displaystyle\lim_{n \rightarrow +\infty }\displaystyle\int_{\Lambda} \lambda (h \circ G) \cdot e^{-\phi}\cdot \chi_{A} \ d\nu.
\end{eqnarray*}
Then we have
      $$ \mu\big(G(A)\big) = \displaystyle\int_{A} \lambda \cdot \frac{(h \circ G)}{h} \cdot e^{-\phi} \ d\mu,$$  
which finishes the proof of the Lemma.                      
\end{proof}

Consider $\alpha$ a partition of $\Lambda$. We define from this a new partition for each $n\in \mathbb{N}$:
            $$\alpha_{n} = \bigvee_{j = 0}^{n-1} G^{-j}(\alpha)  . $$
Namely, an element $\mathcal{P}\in \alpha_{n} $ is written by:
             $$\mathcal{P}= P_{i_0}\cup G^{-1}(P_{i_1})\cup \cdots \cup  G^{-1}(P_{i_{n-1}}) $$ 
where $P_{i_k} \in \alpha$ for every $i_{k}=0,\cdots, n-1$.                        

Fix $c>0$ as in equation~\eqref{condicao sobre c} and $\epsilon$ as in remark~\ref{epsilon}. 
\begin{Lema}\label{particao geradora}
Let $\eta$ be a invariant probability measure whose Lyapunov exponents are bigger than $8c$, and $\alpha$ a partition with diameter smaller than $\epsilon$. Then, for $\eta$ almost every point $x\in \Lambda$, $\alpha_n(x)$ goes to zero when $n \to \infty$. In particular, $\alpha$ is a $G$-generating  partition with respect to $\eta$.
\end{Lema}
The proof is analogous to Lema 4.3 in  \cite{OV05}.
We want to verify that $\mu = h \nu$ has Lyapunov exponents as in the Lemma. 
\begin{Lema}
Let $(a_n)\in \Sigma_{A}$ be a recurrent sequence for the shift $\sigma$ with $a_0 = 2$. Then there exist  $k \in (0,1)$ and an increasing subsequence of times $(a_{n_j})_{j\geq 0}$ such that
       $$ |DG_{|E^c}^{n_j}(x)| \geq C. k^{j}  $$
where $x$ belongs to the cylinder of length ${n_j}$ generated by the first $n_j$ terms of the sequence $(a_n)$.
\end{Lema}
This statement is adapted from Lemma 3.2 in \cite{LOR} and its proof is analogous.

\begin{Lema}\label{expoentes positivos}
The Lyapunov exponents of $\mu = h\nu$ are bigger than $8c$.
\end{Lema}

Now we state a version of the Roklin's Formula proved by Oliveira and Viana in \cite{OV05}. To compute the metric entropy of $G$ with respect to some measure by integrating the $\log$ of  Jacobian of this measure with respect to $G$. See \cite{OV05} for the proof.

Recall the Roklin's Formula: if $\eta$ is an invariant probability measure and there exists a $G$-generated partition with respect to $\eta$ then  we have:
   $$ h_{\eta}(G)= \int \log (J_{\eta}G) \, d\eta   .$$

\begin{Prop}\label{log de lambda}
If $\lambda$ is the spectral radius of the transfer operator $\mathcal{L}_{\phi}$ then the following holds:
        $$h_{\mu}(G) + \int \phi \ d\mu  = \log \lambda = P.$$
\end{Prop}
\begin{proof}
 Lemma~\ref{expoentes positivos} and  Lemma~\ref{particao geradora} imply that there exists a $G$-generating partition. Then, according to the Roklin's Formula we have
  $$h_{\mu}(G) = \int \log (J_{\mu} G)\ d\mu .$$
On the other hand, Lemma~\ref{jacobiano de mu} shows that $J_{\mu G}= \lambda\frac{h\circ G}{h}e^{-\phi}$, which gives us:
\begin{eqnarray*}
 h_{\mu}(G) + \int \phi \ d\mu &=&\int \log (J_{\mu} G)\ d\mu + \int \phi \ d\mu \\
                               &=& \int \log \left(\lambda\frac{h\circ G}{h}e^{-\phi} \right)\ d\mu + \int \phi \ d \mu \\  
                               &=& \int \log \lambda \ d\mu + \int \log(h \circ G)\ d\mu - \int \log h \ d\mu \\
                               &=& \int \log \lambda  \ d\mu = \log \lambda = P. 
\end{eqnarray*} 
The fourth equality follows from the fact that the probability measure is invariant by $G$.
\end{proof}
 
We have to check that $\mu$ is an equilibrium state for the system $(G, \phi)$. Actually, any expanding non-lacunary Gibbs measure is an equilibrium state. This proof was given in \cite{OV08} and still holds in our context. We will give the idea here. We  first remember that 
  $$\sup \{  h_{\eta}(G) + \int \phi \ d\eta \} = P_{top}(G, \phi) = \sup \{ P_{top}(G, \Lambda \backslash \mathcal{H}),  P_{top}(G, \mathcal{H})  \}$$
where the last equality is a result due Pesin in \cite{Pesin97}, and $ P_{top}(G, \Lambda \backslash \mathcal{H})$
denotes the topological pressure relative to the set $\Lambda \backslash \mathcal{H}$.
Then, using the last proposition it is enough to prove that  $P >  P_{top}(G, \Lambda \backslash \mathcal{H}) $ and $P = P_{top}(G, \mathcal{H}) $. 
To see that $P$ coincides with the topological pressure relative to the set $\mathcal{H}$, we construct a family of partitions $\mathcal{U}^m$ of $\mathcal{H}$ formed by hyperbolic cylinders and whose diameters go to zero as $m \to \infty$. If $U \in \mathcal{U}^m$ is a hyperbolic cylinder of lenght $n$, then, by Proposition~\ref{mu Gibbs}, we have
 $$ K^{-1} \leq \displaystyle\sum_{U\in \mathcal{U}} \exp(S_{n}\phi(U) - Pn) \leq K.$$
Note that changing $P$ for an arbitrary $\alpha$, $\alpha \in \mathbb{R}$, in the middle term of the inequality it fails to be true for every $n$ if  $\alpha \neq P$.
Then by the definition of topological pressure relative to a set, we have $P = P_{top}(G, \mathcal{H}) $. 

%%%%%%%%%%%%%%%%%%%%%%%%%%%%%%%%%%%%%%%%%%%%%%%%%%%%%%%%%%%%%%%%%%%%%%%%%%%%%%%%%%%%%%%%%%%%%%%%%%%%%%%%%%%%%%%%
%%%%%%%%%%%%%%%%%%%%%%%%%%%%%%%%%%%%%%%%%%%%%%%%%%%%%%%%%%%%%%%%%%%%%%%%%%%%%%%%%%%%%%%%%%%%%%%%%%%%%%%%%%%%%%%%
%%%%%%%%%%%%%%%%%%%%%%%%%%%%   UNICIDADE                       %%%%%%%%%%%%%%%%%%%%%%%%%%%%%%%%%%%%%%%%%%%%%%%%
%%%%%%%%%%%%%%%%%%%%%%%%%%%%                                    %%%%%%%%%%%%%%%%%%%%%%%%%%%%%%%%%%%%%%%%%%%%%%%%
%%%%%%%%%%%%%%%%%%%%%%%%%%%%%%%%%%%%%%%%%%%%%%%%%%%%%%%%%%%%%%%%%%%%%%%%%%%%%%%%%%%%%%%%%%%%%%%%%%%%%%%%%%%%%%%%
%%%%%%%%%%%%%%%%%%%%%%%%%%%%%%%%%%%%%%%%%%%%%%%%%%%%%%%%%%%%%%%%%%%%%%%%%%%%%%%%%%%%%%%%%%%%%%%%%%%%%%%%%%%%%%%%

\section{Uniqueness}\label{Uniqueness}

In this section we finish the proof of  Theorem~\ref{Teorema Principal} proving the uniqueness of the equilibrium state. 
Applying the Jensen's Formula from basic calculus to the logarithm function we have the following lemma.

\begin{Lema}\label{Jensen}
Let $p_{i}$ and $q_{i}$ be positive real numbers for each $i=1, \cdots, n$.
The next inequality holds:
         $$\frac{\displaystyle\sum_{i=1}^{n} p_{i}\log q_{i}}{\sum p_{j}} \leq \log\left( \frac{\displaystyle\sum_{i=1}^{n} p_{i}q_{i}}{\sum p_{j}} \right) .   $$
Moreover, the equality holds if, and only if, the $q_{i}$'s are all equal for every $i=1, \cdots, n$. 

\end{Lema}
\begin{Prop}\label{jacobiano est eq}
Let $\eta$ be an equlibrium state for the system $(G,\phi)$ of the Theorem~\ref{Teorema Principal}. The Jacobian of $\eta$ is given by 
      $$J_{\eta}G= \frac{h \circ G}{h} \lambda e^{-\phi}.$$
\end{Prop}

\begin{proof}
Since $\mu= h\nu$ is an equilibrium state, (according to the previous section) it maximizes the topological pressure. If $\eta$ is also an equilibrium state, we have:
   \begin{eqnarray*}
   h_{\eta}(G) + \int \phi \, d\eta = h_{\mu}(G) + \int \phi \, d\mu = \log \lambda.        
   \end{eqnarray*}
Note that the last equality follows from the Lemma~\ref{log de lambda}.
Define the function $g$ on $\Lambda$ by: 
         $$g(x)= \frac{h(x)e^{\phi(x)}}{\lambda ( h \circ G)(x)} .$$
We want to show that $J_{\eta}G = \frac{1}{g}$.       
By the definition of $g$ we have:
  \begin{eqnarray*}
   h_{\eta}(G) + \int \log g \, d\eta & = & h_{\eta}(G) + \int \log \Big(\frac{h(x)e^{\phi(x)}}{\lambda (h \circ G)(x)}\Big) \, d\eta  \\
                                      & = & h_{\eta}(G) + \int \phi \, d\eta - \log \lambda +\\
  &+&\int \log(h) - \log(h \circ G)\, d\eta  \\
                                      & = & \int \log(h) - \log(h \circ G)\, d\eta = 0.
   \end{eqnarray*}
We point out that the last equality follows from the invariance of $\eta$.
On the other hand, using the Roklin's Formula we have:
  $$h_{\eta}G =\int \log J_{\eta}G \ d\eta. $$
Replacing it in the previous equation, we have:  
\begin{eqnarray*}
   0& = & h_{\eta}(G) + \int \log g \, d\eta = \int \log \left( \frac{g}{(J_{\eta}G)^{-1}} \right) \, d\eta \\
    & = & \int \displaystyle\sum_ {y \in G^{-1}(x)}( J_{\eta}G(y))^{-1}\log \left( \frac{g(y)}{(J_{\eta}G(y))^{-1}} \right) \, d\eta.  
\end{eqnarray*}  

Since $\eta$ is $G$-invariant it yelds :
\begin{equation}\label{pi eh um}
\displaystyle\sum_ {y \in G^{-1}(x)} \frac{1}{J_{\eta}G(y)}= 1,   \quad \eta-\mbox{q.t.p} \ x \in \Lambda.
\end{equation}
Note that any $x\in \Lambda$ has, at most, 2 pre-images. Let $f_n(x)$ be the number of pre-images of $x$ by $G^n$. Consider $y_{i}\in G^{-n}(x)$ with $i=1,\cdots,f_n(x)$ and take:  
     $$ p_{i}= \frac{1}{J_{\eta}G(y_{i})}  \quad \mbox{e} \quad  q_{i}= \frac{g(y_i)}{(J_{\eta}G)^{-1}(y_{i})}      .   $$

By equation~\eqref{pi eh um}, we have $\sum p_i =1$. Applying the Lemma~\ref{Jensen} we have
 \begin{eqnarray*}
  \displaystyle\sum_{i=1}^{f_n(x)} p_{i}\log q_{i} & \leq & \log \big( \displaystyle\sum_{i=1}^{f_n(x)} p_{i}q_{i}\big)\\
          										& =    & \log \big(\displaystyle\sum_{i=1}^{f_n(x)} g(y_{i})\big)= \log 1 = 0 .	
\end{eqnarray*}
Integrating with respect to $\eta$:

$0 = \displaystyle\int \sum_{i=1}^{f_n(x)} p_{i}\log q_{i} \, d \eta(x) \leq \int  \log \Big(\displaystyle\sum_{i=1}^{f_n(x)} p_{i} q_{i}\Big) \, d \eta(x) = 0 .$

Then, the equality holds. So  Lemma~\ref{Jensen} also guarantees that the  $q_{i}$'s are equal. Thus, there exists a constant which depends on $x$, such that 
    $$q_{i} = \frac{g(y)}{(J_{\eta}G(y))^{-1}}= c(x) \quad \forall y \in G^{-1}(x)  .$$

By \eqref{pi eh um}, again, we have:
      $$1 = \frac{\displaystyle\sum_{y \in G^{-1}(x)}g(y)}{\displaystyle\sum_{y \in G^{-1}(x)} (J_{\eta}G(y))^{-1}} =  c(x) .    $$
Then, $g$ coincides with $(J_{\eta}G)^{-1}$  in every pre-image of points in a $\eta$ full measure set. Finally, the invariance of $\eta$ gives that they coincide in $\eta$ almost every point. 
\end{proof}
From the Lemma~\ref{h positivo} we have that $h>0$. So, the measure given by $\tau = \frac{1}{h}\eta$ is well-defined. We will see that this measure is an eigenmeasure for the dual transfer operator.

\begin{Lema}\label{est eq automedida}
Let $\eta$ be an equilibrium state for the system $(G,\phi)$. Then the measure $\tau = \frac{1}{h}\eta$ is an eigenmeasure for $\mathcal{L}_{\phi}^{\ast}$ associated to the spectral radius of $\mathcal{L}_{\phi}$. In other words, $\tau$ satisfies: 
         $$\mathcal{L}_{\phi}^{\ast}\tau = \lambda \tau. $$
\end{Lema}

\begin{proof}
We will show that every continuous function $\psi : \Lambda \rightarrow \mathbb{R}$ satisfy 
         $$ \int \psi \  d \mathcal{L}_{\phi}^{\ast}\tau =  \lambda \int \psi \ d\tau  .    $$
         
In fact,
\begin{eqnarray*}
\displaystyle\int \psi \  d \mathcal{L}_{\phi}^{\ast}\tau &=& \displaystyle\int \mathcal{L}_{\phi} \psi \ d \tau \\
                                                          &=& \displaystyle\int \mathcal{L}_{\phi}\psi\cdot \frac{1}{h} \ d \eta \\
                                                          &=& \displaystyle\int \sum_{y \in G^{-1}(x)}\Big[ e^{\phi(y)}\psi(y)\cdot \frac{1}{h \big(G(y)\big)}\Big] \ d \eta(x). 
\end{eqnarray*}     

Since $\eta$ is an equilibrium state, the Proposition~\ref{jacobiano est eq} ensures that for every  $y$ the  Jacobian of $\eta$ satisfies: 
\begin{eqnarray}\label{eq jacobiano est eq}
   \frac{e^{\phi(y)}}{h(G(y))} = \lambda \cdot \frac{1}{J_{\eta}G(y)h(y)}.  
\end{eqnarray}
Thus:
\begin{eqnarray*}
   \displaystyle\sum_{y \in G^{-1}(x)} \frac{1}{J_{\eta}G(y)} &= &\displaystyle\sum_{y \in G^{-1}(x)}\frac{ e^{\phi(y)}\cdot h(y)}{\lambda h(G(y))} \\
   = \frac{ \displaystyle\sum_{y \in G^{-1}(x)}e^{\phi(y)}\cdot h(y)}{\lambda h(x)}
 &=&\frac{ \mathcal{L}_{\phi}h(x)}{\lambda h(x)} = 1.  
\end{eqnarray*} 
The last equality comes from the fact that $h$ is an eigenfunction of $\mathcal{L}_{\phi}$.  
 According to the equation~\eqref{eq jacobiano est eq}, we have:
\begin{eqnarray*}
\displaystyle\int \psi \  d \mathcal{L}_{\phi}^{\ast}\tau &=& \displaystyle\int \sum_{y \in G^{-1}(x)}\lambda \frac{\psi(y)}{J_{\eta}G(y)h(y)} \ d \eta(x)\\
                                                          &=& \lambda \displaystyle\int \sum_{y \in G^{-1}(x)}\frac{1}{J_{\eta}G(y)} \frac{\psi(y)}{h(y)} \ d \eta(x)\\
                                                          &=&\lambda \displaystyle\int \displaystyle\sum_{y \in G^{-1}(x)} \frac{\psi(y)}{h(y)} \ d \eta \\  
                                                          &=& \lambda \displaystyle\int \frac{\psi }{h}\ d \eta = \lambda \displaystyle\int \psi \ d \tau.                      
\end{eqnarray*}         
wich finishes the proof.         
\end{proof}
As consequence of this result, any equilibrium state associated to $(G, \phi)$ is equivalent to an eigenmeasure and, thus, is an expanding measure: it gives full measure to the set $\mathcal{H}$.

\begin{Prop} \label{eq. eh gibbs}
Consider the system $(G,\phi)$ as in the Theorem~\ref{Teorema Principal}.
Any equilibrium state of $G$ associated to the potential $\phi$ is a non-lacunary Gibbs probability measure.
\end{Prop}

\begin{proof}
Let $\eta$ be an equilibrium state. By the Lemma~\ref{est eq automedida} the measure $\tau = \frac{1}{h}\eta$ is an eigenmeasure of the dual operator, $\mathcal{L}_{\phi}^{\ast}$, associated to the eigenvalue $\lambda$. Therefore, $\tau$ is a non-lacunary Gibbs measure. 
On the other hand, we can write $\eta= h\tau$. So, the Proposition~\ref{Gibbs} guarantees that $\tau$ is a non-lacunary Gibbs measure.
\end{proof}

\begin{Lema}\label{Gibbs equivalentes}
Let $\mu_ {1} $ and $ \mu_{2}$ be  expanding non-lacunary Gibbs probability measures.
Then $\mu_{1}$ and $\mu_{2}$ are equivalent.
\end{Lema}

This result is proven in \cite{OV08}. In our context, it follows easier since the rectangles have empty boundaries.
%
%\begin{proof}
%Consider $R^{n}$ a cylinder of length $n$.
%By the definition of non-lacunary Gibbs measure, we have:
%    $$K^{-1}\leq \frac{\mu_{i}(R^{n})}{\exp(S_{n}\phi(x)- nP)} \leq K , \quad i=1,2  \ \mbox{and for all} \ x \in R^{n}     .$$ 
%The,
%\begin{eqnarray*}
%K^{-2} \mu (R^{n}) &\leq & K^{-2} \cdot K \exp(S_{n}\phi(x)- nP) \\
%                   &\leq & K^{-1} \exp (S_{n}\phi(x)- nP) \\
%                   &\leq & \mu_{1}(R^{n}) \\
%                   &\leq & K \exp (S_{n}\phi(x)- nP)\\
%                   &\leq & K^{2}\cdot K{-1} \exp(S_{n}\phi(x)- nP)\\
%                   &\leq & K^{2} \mu_{1}(R^{n}). 
%\end{eqnarray*} 
%If we put $M = K^{2}$, we have:
%    $$ M^{-1} \mu_{1}(R^{n}) \leq \mu_{2}(R^{n}) \leq M \mu_{1}(R^{n}).   $$
%   
%\end{proof}  
%\marginpar{escrever uma conclusao decente}

We are now able to finish the proof of the Theorem~\ref{Teorema Principal}. Let $\eta$ be an ergodic equilibrium state for the system $(G,\phi)$. Then $\eta $ is expanding, and  Proposition~\ref{eq. eh gibbs} ensures that $\eta$ is a non-lacunary Gibbs invariant measure. On the other hand, the probability measure $\mu=h\nu$ (where $h$ and $\nu$ are given by the Theorem~\ref{Teorema B} is also expanding and an invariant non-lacunary Gibbs  measure. But the Lemma~\ref{Gibbs equivalentes} ensures that two expanding non-lacunary Gibbs probability measures are equivalent. Then $\mu$ is ergodic and we have $\mu = \eta $.
This concludes the proof of Theorem~\ref{Teorema Principal}.

%%%%%%%%%%%%%%%%%%%%%%%%%%%%%%%%%%%%%%%%%%%%%%%%%%% %%%%%%%%%%%%%%%%%%%%%%%%%%%%%%%%%%%%%%%%%%%%%%%
%%%%%             EQ PARA HORSESHOES                                                                     %%%%%%%%%%%%%%%%%%%%   
%%%%%%%%%%%%%%%%%%%%%%%%%%%%%%%%%%%%%%%%%%%%%%%%%%%%
%%%%%%%%%%%%%%%%%%%%%%%%%%%%%%%%%%%%%%%%%%%%%%%%%%%%%%%%%%%%%%%%%%%%%%%%%%%%%%%%%%%%%%%%%%%%%%%%%%%%%%%%%%%%%

\section{Equilibrium states for the Horseshoe} \label{back}

%Define $\mu^{\star}$ by:
%     $$\mu^{\star}(A) = \mu \phi (\pi (A))   \quad \forall A \ \mbox{Borelian in }$$

In this section we prove Theorem~\ref{Teo C}.
Let $G: \Lambda \to \Lambda$ be the map defined in the section~\ref{main theorems}.
Consider also the high dimension Horseshoe $F: \Omega \to \Omega$ also defined in section~\ref{main theorems} and consider its inverse $F^{-1}$.

We define a projection of the parallelepipeds $\tilde{R}_0$ and $\tilde{R}_1$ onto the planes $P_0$ and $P_1$, $\pi:\tilde{R}_1 \cup \tilde{R}_2  \to P_0 \cup P_1$ by:
$$\pi : (x,y,z)= \left\{ \begin{array}{rc}
(x,y,0),  &\mbox{if}  \quad (x,y,z) \in \tilde{R}_0 \\
(x,y,\frac{5}{6}),  &\mbox{if}  \quad (x,y,z) \in \tilde{R}_1
\end{array}
\right.
$$
It is straightforward to check that $\pi$ is continuous, surjective and  $$\pi \circ F^{-1}= G \circ \pi.$$

\begin{Lema}\label{zero na fibra}
For each $X \in \Lambda$ we have $h(F^{-1}, \pi^{-1}(X))= 0.$
\end{Lema}
\begin{proof}
Since the inverse horseshoe contracts in the Oz direction, for each $X \in \Lambda$ its restriction to a fiber $\pi^{-1}(X)$,

 $F^{-1}|_{\pi^{-1}(X)}: \pi^{-1}(X) \to \pi^{-1}(G(X))$ is a contraction.

The number $S(n,\epsilon,\pi^{-1}(X))$ which is de maximum of the cardinalities of the sets $(n,\epsilon)$-separated in $\pi^{-1}(X)$ is constant and depends only on $\epsilon$ not on $n$. Thus
   $$h(F^{-1}, \pi^{-1}(X))= \displaystyle\lim_{\epsilon \to 0} \displaystyle\lim_{n \to 0} \frac{1}{n}\log S(n,\epsilon, \pi^{-1}(X))= 0.$$
\end{proof}

Since the projection $\pi$ is actually a semiconjugacy between the inverse Horseshoe $F^{-1}$ and the map $G$, we have [See \cite{Bowen71}]:
  $$ h_{top}(F^{-1}) \leq h_{top}(G) +\sup \left\{h(F^{-1}, \pi^{-1}(X)), X \in \Lambda \right\}. $$
By the Lema~\ref{zero na fibra} we get:
      $$h_{top}(F^{-1}) \leq h_{top}(G).  $$
On the other hand, because $\pi$ is a semiconjugacy we have the other inequality imediately (see \cite{Bowen75}), wich gives us:
\begin{equation}\label{igualdade das entropias}
h_{top}(F^{-1}) = h_{top}(G).
\end{equation}
Thus we proved that the topological entropy of these two maps are equal.

\subsection{Existence}\label{Existence}
Consider the system  $(G,\phi)$, with $G$ as defined in the section~\ref{main theorems} and $\phi$ a potential  H\"older continuous with small variation:
    $$ \sup \phi - \inf \phi < \frac{log(\omega)}{2}.  $$
Let $\mu_{\phi}$ be the unique equilibrium state associated to the system $(G,\phi)$ obtained in Theorem~\ref{Teorema Principal}. Let $\mathcal{A}_{\Lambda}$ be the Borel $\sigma$-algebra in $\Lambda$ and define $\mathcal{A}_0:= \pi^{-1}({A_{\Lambda}})$ .

Thus, $\mathcal{A}_0$ is a $\sigma$-algebra  in the Horseshoe, whose elements are given by $A=\pi^{-1}(B)$, where $B$ is a Borelian.
Notice that $\mathcal{A}_0 \subset F^{-1}(\mathcal{A}_0)$. 
If for each $n\in \mathbb{N}$ we define $\mathcal{A}_n:= F^{-n}(\mathcal{A}_0) $ we have a increasing sequence of $\sigma$-algebras:
             $$ \mathcal{A}_0 \subset \mathcal{A}_1 \subset \mathcal{A}_2 \subset \cdots \subset \mathcal{A}_n \subset \cdots  $$
For each $n$ define a probability measure $\mu^{\star}_{n}: \mathcal{A}_n \to [0,1]$ by
         $$ \mu^{\star}_{n}(F^{-n}(A_0))= \mu_{\phi}(\pi(A_0)) \quad \mbox{for all} \ A_0 \in \mathcal{A}_0. $$
Notice that $\mu^{\star}_{n}$ depends also on $\phi$.
Using the semiconjugacy it is straightforward to prove that each $\mu^{\star}_{n}$ is invariant by $F^{-1}$.
Finally, we consider $\mathcal{A}:= \displaystyle\cup_{n=0}^{\infty}\mathcal{A}_n$ a $\sigma$-algebra in $\Omega$.

Define $\mu^{\star}: \mathcal{A} \to [0,1]$ by
   $$\mu^{\star}(A)= \mu^{\star}_n (A)   \quad \mbox{if} \  A \in \mathcal{A}_n  . $$
It is easy to see that $\mu^{\star}$ is well-defined. We have to prove that $\mu$ is  $\sigma$-additive. For this we will use the following criterium from classical measure Theory.

\begin{Prop}
Let $\mathcal{A}$ be a $\sigma$-algebra on $\Omega$ and $\mu^{\star}: \mathcal{A} \to [0,1]$  satisfying 
   $$\mu^{\star}(A_1 \cup \cdots \cup A_n ) = \sum_{i=1}^{n}\mu^{\star}(A_i) $$
for every finite family $(A_i)$ of disjoints sets on $\mathcal{A}$. If there exists a family $\mathcal{C} \subset \mathcal{A}$ which satisfies:
\begin{enumerate}
\item $\mathcal{C}$ is a compact class:  if $C_1 \supset C_2 \supset \cdots \supset C_n \supset \cdots$  are elements of the family $\mathcal{C}$ then $\bigcap_{n=1}^{\infty}C_n \neq \emptyset$.\\
\item $\mathcal{C}$ has the approximation property:  for all $A \in \mathcal{A}$ we have
$$\mu^{\star}(A)= \sup\{ \mu^{\star}(C) | C \subset A, C \in \mathcal {C}  \}.$$
\end{enumerate}
then $\mu^{\star}$ is a probability mesure on $\mathcal{A}$.
\end{Prop}
We will construct a family $\mathcal{C}$ which satisfies the conditions (i) and (ii) in Proposition above.

We consider $K_{\Lambda}$ the collection of all compacts on $\Lambda$ and we define $K_0 := \pi ^{-1}(K_{\Lambda})$ a subfamily of the $\sigma$-algebra $\mathcal{A}_0$. 

We define $\mathcal{C}:= \bigcup_{n=0}^{\infty} F^{-n}(K_0) $ which is contained in $\mathcal{A}$ and is a compact class with the approximation property:
\begin{enumerate}
\item Let  $C_1 \supset C_2 \supset \cdots \supset C_n \supset \cdots$ be sets of $\mathcal{C}$. For every $i \in \mathbb{N}$ there exists $D_i \subset K_{\Lambda}$ and there exists $n_i \in \mathbb{N}$ such that $C_i = F^{-n_i}(\pi^{-1}(D_i))$. Since $\pi^{-1}(D_i)$  is a closed set on a compact space $\Lambda$,  $\pi^{-1}(D_i)$ is compact. By the continuity of $F^{-1}$, $(C_i)_{i \in \mathbb{N}}$ is a family of nested compact sets, therefore $\bigcap_{n=1}^{\infty}C_n \neq \emptyset$.\\
\item Let $A \in \mathcal{A}$. Using the regularity of  the measure $\mu_{\phi}$  on $\Lambda$ we have $\mu_{\phi}(B)= \sup\{ \mu_{\phi}(D) | D \subset B, D \in K_{\Lambda}  \}$ for all $D \in \mathcal{A}_{\Lambda}$.
By the definition of $\mathcal{C}$,  for all $C \in \mathcal{C}$ there exist $D \in K_{\Lambda} $ and $n \in \mathbb{N}$ such that $C= F^{-n}(\pi^{-1}(D))$.
Thus
$ \mu^{\star}(C)  = \mu^{\star}(F^{-n}(\pi^{-1}(D))) = \mu_{\phi}(\pi( \pi^{-1} (D)) = \mu_{\phi}(D). $ Then
$$\sup\{ \mu_{\phi}(D) | D \subset B, D \in K_{\Lambda}  \} = \sup\{ \mu_{\star}(C) | C \subset A, C \in \mathcal{C}  \}.$$ 
To finish it is enough to notice that for all $ A \in \mathcal{A}$ there exists $B \in K_{\Lambda}$ such that $\mu^{\star}(A) =\mu_{\phi}(B) $.
\end{enumerate}
Then, by the criterium we proved that $\mu^{\star}$ is a probability measure on $\mathcal{A}$.
Since every $\mu_n^{\star}$ is invariant by $F^{-1}$ the measure $\mu^{\star}$ is also  $F^{-1}$-invariant.
\begin{Lema}
The  extension of the $\sigma$-algebra $\mathcal{A}$ coincides with the Borel $\sigma$-algebra on $\Omega$.
\end{Lema}
\begin{proof}
Since $\pi$ is continuous we have that $\mathcal{A}$ is contained in the Borel $\sigma$-algebra. It is enough to prove that for all $ X \in \Omega$ there exists a fundamental system of neighborhoods of $X$ contained in $\mathcal{A}$.

Let $X \in \Omega$ and remember that $\rho < 1$ is the factor of contraction in the direction of the $z$-axis. For each $n \in \mathbb{N}$, by the uniform continuity of $G^n$ there exists $\delta_n$ such that $d(z, w) < \delta_n $ implies $d(G^n(y),G^n (z)) < \frac{1}{n}$. 

We define $B_n(X)= F^{-n}(\pi^{-1}(\mathcal{B}(\pi (F^{n}(X)), \delta_n)))$.
Therefore, for every $n \in \mathbb{N}$ and $X \in B_n(X)$, we have 
$diam B_n(X)= \rho^n + \frac{1}{n}$
which goes to zero when $n$ goes to infinity.
\end{proof}

\begin{Lema}\label{des entropia}
We have $ h_{\mu^{\star}}(F^{-1})  \geq  h_{\mu_\phi}(G)$.
\end{Lema}
\begin{proof}
Using the well known Brin-Katok Theorem we can compute the entropy by the formula [see \cite{Walters}]:
$$  h_{\mu_{\phi}}(G) = \lim_{\epsilon \to 0 }\limsup_{n \to +\infty} \frac{1}{n} \log \frac{1}{\mu_{\phi}( {\mathcal{B}}^{n}_{G}(X, \epsilon) )} $$
for almost every $X$ on $\Lambda$ with respect to $\mu_{\phi}$, where $ {\mathcal{B}}^{n}_{G}(X, \epsilon) $ denotes the dynamical ball of $G$ centered on $X$ and with radius $\epsilon$ and lenght $n$.
We notice that, given $\epsilon>0$, by the continuity of $\pi$, (on a compact space) there exists $0\leq \delta \leq \epsilon$ such that
 $$\pi(\mathcal{B}(Z,\delta)) \subset\mathcal{B}(\pi(Z),\epsilon) \quad \mbox{for every} \quad Z \in \Omega .$$
Using this fact it is straightforward to check that
   $$ {\mathcal{B}}^{n}_{F^{-1}}(Y, \delta)    \subset\pi^{-1}( {\mathcal{B}}^{n}_{G}(X, \epsilon)) \quad \mbox{for every} \quad Y \in \pi^{-1}(X) . $$
By the definition of the measure $\mu^{\star}$ we have: 
         $$\mu^{\star}( {\mathcal{B}}^{n}_{F^{-1}}(Y, \delta)  )  \leq    \mu^{\star}(\pi^{-1}( {\mathcal{B}}^{n}_{G}(X, \epsilon))) = \mu_{\phi}( {\mathcal{B}}^{n}_{G}(X, \epsilon))  . $$

If we consider $M \subset \Lambda$ a full measure set with respect to $\mu_{\phi}$ then $\pi^{-1}(M) \subset \Omega$ satisfying $\mu^{\star}(\pi^{-1}(M)) =\mu_{\phi}(M)=1 $. Thus:
$$  h_{\mu_{\phi}}(G) \leq \lim_{\epsilon \to 0 }\limsup_{n \to +\infty} \frac{1}{n} \log \frac{1}{\mu_{\phi}( {\mathcal{B}}^{n}_{G}(X, \epsilon) )} = h_{\mu^{\star}}(F^{-1}). $$
\end{proof}

Consider $\varphi:= \phi \circ \pi$. Thus $\varphi$ is a real function in $\Omega$.
\begin{Prop}
Given the system $(F^{-1}, \varphi)$ the probability measure $ \mu^{\star}$  is an equilibrium state associated to it.
\end{Prop}
\begin{proof}
Let $\overline{\mu}$ be an $F^{-1}$-invariant probability measure on $\Omega$. Consider a probability measure on $\Lambda$, defined by $\eta :=\overline{\mu} \circ \pi ^{-1} $. We have:  
$$h_{\overline{\mu}} (F^{-1}) + \int \phi \circ \pi \, d\overline{\mu} \leq h_{\eta}(G) + \int \phi \, d{\eta} +\displaystyle\int_{\Lambda} h_{top}(F^{-1}, \pi^{-1}(X)) d\mu(X).$$
It follows from a result due Ledrapier and Walters \cite{Ledra}. For a proof of the statement see Theorem 4.1 of \cite{AP}.
According to Lema~\ref{zero na fibra} we have
$$\displaystyle\int_{\Lambda} h_{top}(F^{-1}, \pi^{-1}(X)) d\mu(X) = 0.$$
 By the semiconjugacy, $\eta$ is $G$-invariant, thus
  $$ h_{\eta}(G) + \int \phi \, d\eta  \leq P_{top} (G,\phi)  $$
Going back to the inequality we get:
$$h_{\overline{\mu}} (F^{-1}) + \int \phi \circ \pi \, d\overline{\mu} \leq  P_{top} (G,\phi)  $$
Since the last inequality holds for any $F^{-1}$-invariant measure:
      $$ P_{top} (F^{-1},\varphi)\leq  P_{top} (G,\phi). $$
On the other hand, $\pi$ is a semiconjugacy, which implies [see \cite{Bowen75}] 
        $$ P_{top}(G,\phi)\leq  P_{top}  (F^{-1},\varphi) .$$
Therefore the topological pressure of both systems coincide.
Notice that, by change of variables, we get
  $$\int \phi \, d\mu_{\phi}= \int \phi \circ \pi \, d\mu^{\star}= \int \varphi \, d\mu^{\star}. $$
Using the fact that  $h_{\mu_{\star}}(F^{-1}) \geq h_{\mu_{\phi}}(G)  $, proved in Proposition~\ref{des entropia}, and that the topological pressures coincides we have:
$$h_{\mu^{\star}} (F^{-1}) + \int \phi \circ \pi \, d\mu^{\star} \geq  h_{\mu_{\phi}}(G)  + \int \phi \, d\mu_{\phi} =  P_{top} (F^{-1},\varphi).$$
Thus $\mu^{\star}$ is an equilibrium state associated to $(F^{-1}, \varphi)$.
\end{proof}

\subsection{Uniqueness}

In the last section we constructed a probability measure provenient of the equilibrium state on the space $\Lambda$. We proved that this measure is actually an equilibrium state for the horseshoe associated to a potential which projects onto the potential on $\Lambda$.

Consider the equilibrium state $\mu^{\star}$ for $(F^{-1}, \varphi)$ obtained in the section~\ref{Existence}. Assume that there is another equilibrium state $\eta^{\star}$ which is different of  $\mu^{\star}$. 
We define a probability measure on $\Lambda$ by the push-forward of  $ \eta^{\star} $:
 $$\eta (A):=\eta^{\star}(\pi^{-1}(A)) \quad \mbox{for all} \ A \ \mbox{Borelian set.}  $$
We will prove that the measure $\eta$ differs of the unique equilibrium state on $\Lambda$, given by  Theorem~\ref{Teorema Principal}.
\begin{Lema}
The probability measure $\eta$ is different of the equilibrium state $\mu_{\phi}$.
\end{Lema}
\begin{proof}
We suppose that $\eta^{\star} \neq \mu^{\star} $ which implies that there exists an $A \in \mathcal{A}$ satisfying $\eta^{\star}(A) \neq \mu^{\star}(A)$ .
The $\sigma$-algebra $\mathcal{A}$ , wich coincides with the Borelians on $\Lambda$, was obtained as $\mathcal{A}= \mathcal{A}_0 \cup F^{-1}\mathcal{A}_0 \cup \cdots \cup  F^{-n}\mathcal{A}_0 \cup \cdots $ 
Thus there exists $A_0 \in \mathcal{A}_0$ and $n_0 \in \mathbb{N}$ such that  $A= F^{-n_0}(A_0)$.

Remember also that  $\mathcal{A}_0=\pi^{-1}\left(  \mathcal{A}_\Lambda \right)$ and therefore there exists $B \in \mathcal{A}_\Lambda$ such that $A_0 = \pi^{-1}(B)$.
Using that $\mu^{\star}$ is invariant by $F^{-1}$ we have:
$$ \eta(B)= \eta^{\star}(\pi^{-1}(B)) = \eta^{\star}(A_0) =  \eta^{\star}(F^{-n_0}(A_0)) =  \eta^{\star}(A). $$
On the other hand, we have:
$$ \mu_{\phi}(B)= \mu_{\phi}(\pi(A_0)) = \mu^{\star}(A_0)= \mu^{\star}(F^{-n_0}A_0) =  \mu^{\star}(A).  $$

Since $\eta^{\star}(A) \neq \mu^{\star}(A) $  we have that $ \eta(B) \neq  \mu_{\phi}(B)$  and the Lemma is proved.
\end{proof}

By supposing the existence of another equilibrium state for the Horseshoe we defined a probability measure on $\Lambda$, denoted by $\eta$, that does not coincide with the unique equilibrium state on $\Lambda$. We are going to prove that $\eta$ is an equilibrium state on $\Lambda$ which contradicts the uniqueness of equilibrium states to the system $(G,\phi)$ and therefore we have proved that there is only one equilibrium state to the Horseshoe associated to the potential $\varphi$.

\begin{Lema}
The probability measure $\eta$ is an equilibrium state associated to the system $(G, \phi).$
\end{Lema}
\begin{proof}
We have to prove that:
$$  h_{\eta}(G)  + \int \phi \, d\eta \geq  P_{top}(G, \phi). $$
In the section~\ref{Existence} we proved that $P_{top}(G, \phi)= P_{top}(F^{-1},\varphi)$. Besides that we proved that  $\eta^{\star}$ is an equilibrium state to $(F^{-1}, \varphi)$. Thus we have:
$$  P_{top}(G, \phi) =  P_{top}(F^{-1},\varphi) =   h_{\eta^{\star}}(F^{-1})  + \int \varphi \, d\eta^{\star} . $$
Then it is enough to prove:
$$  h_{\eta}(G)  + \int \phi \, d\eta \geq  h_{\eta^{\star}}(F^{-1})  + \int \varphi \, d\eta^{\star}. $$
The integrals in the inequality coincide (by a change of variables).  By the Lema~\ref{des entropia} we have $ h_{\eta}(G)  \geq  h_{\eta^{\star}}(F^{-1})$ which finishes the proof.
\end{proof}

\subsection*{Acknowledgement} This work was carried out at Universidade Federal Fluminense, Universit\'e de Bretagne Occidentale, Universidade Federal de Alagoas and Universidade do Porto. The authors are very thankful to Renaud Leplaideur, Krerley Oliveira, Marcelo Viana and Ralph Teixeira for very helpful conversations and encouragement. This work was partially supported by Faperj, Cnpq and the projects DynEurBraz and BREUDS.

\end{document}